\documentclass[reqno,10pt,centertags,draft]{amsart}
\usepackage{amsmath,amsthm,amscd,amssymb,latexsym,upref}
\date{\today}
\newcommand{\bbN}{{\mathbb{N}}}
\newcommand{\bbR}{{\mathbb{R}}}

\newcommand{\bbC}{{\mathbb{C}}}

\newcommand{\cB}{{\mathcal B}}

\newcommand{\cD}{{\mathcal D}}

\newcommand{\cH}{{\mathcal H}}

\newcommand{\cN}{{\mathcal N}}

\newcommand{\cS}{{\mathcal S}}

\newcommand{\cV}{{\mathcal V}}
\newcommand{\cX}{{\mathcal X}}

\newcommand{\dott}{\,\cdot\,}
\newcommand{\no}{\notag}
\newcommand{\lb}{\label}
\newcommand{\f}{\frac}

\newcommand{\ol}{\overline}

\newcommand{\wti}{\widetilde}

\newcommand{\ran}{\text{\rm{ran}}}

\newcommand{\dom}{\text{\rm{dom}}}

\newcommand{\supp}{\text{\rm{supp}}}

\newcommand{\bi}{\bibitem}
\newcommand{\hatt}{\widehat}
\newcommand{\beq}{\begin{equation}}
\newcommand{\eeq}{\end{equation}}
\newcommand{\ba}{\begin{align}}
\newcommand{\ea}{\end{align}}

\newcommand{\abs}[1]{\lvert#1\rvert}

\renewcommand{\Im}{\text{\rm Im}}

\renewcommand{\ge}{\geqslant}
\renewcommand{\le}{\leqslant}

\newcommand{\norm}[1]{\left\Vert#1\right\Vert}

\newcommand{\Om}{\Omega}
\newcommand{\dOm}{{\partial\Omega}}
\newcommand{\si}{\sigma}

\newcommand{\ga}{\gamma}

\newcommand{\eps}{\varepsilon}

\newcommand{\LOm}{L^2(\Om;d^nx)}
\newcommand{\LdOm}{L^2(\dOm;d^{n-1} \omega)}

\allowdisplaybreaks \numberwithin{equation}{section}
\newtheorem{theorem}{Theorem}[section]

\newtheorem{lemma}[theorem]{Lemma}
\newtheorem{corollary}[theorem]{Corollary}
\newtheorem{hypothesis}[theorem]{Hypothesis}
\theoremstyle{definition}

\newtheorem{remark}[theorem]{Remark}

\begin{document}

\title[Robin-to-Robin Maps and Krein-Type Resolvent Formulas]{Robin-to-Robin 
Maps and Krein-Type Resolvent Formulas for Schr\"odinger Operators on Bounded Lipschitz Domains}
\author[F.\ Gesztesy and M.\ Mitrea]{Fritz Gesztesy and Marius Mitrea}
\address{Department of Mathematics,
University of Missouri, Columbia, MO 65211, USA}
\email{fritz@math.missouri.edu}
\urladdr{http://www.math.missouri.edu/personnel/faculty/gesztesyf.html}
\address{Department of Mathematics, University of
Missouri, Columbia, MO 65211, USA}
\email{marius@math.missouri.edu}
\urladdr{http://www.math.missouri.edu/personnel/faculty/mitream.html} 
\thanks{Based upon work partially supported by the US National Science
Foundation under Grant Nos.\ DMS-0400639 and FRG-0456306.}
\dedicatory{Dedicated to the memory of M.\ G.\ Krein (1907--1989).}
%\thanks{To appear in {\it  }}
\date{\today}
%\date{February 27, 2008.}
\subjclass[2000]{Primary: 35J10, 35J25, 35Q40; Secondary: 35P05, 47A10, 47F05.}
\keywords{Multi-dimensional Schr\"odinger operators, bounded Lipschitz domains,
Robin-to-Dirichlet and Dirichlet-to-Neumann maps.}

%%%%%%%%%%%%%%%%%%%%%%%%%%%%%%%%%%%%%%%%
%%%%%%%%%%%%%%%%%%%%%%%%%%%%%%%%%%%%%%%%
\begin{abstract}
We study Robin-to-Robin maps, and Krein-type resolvent formulas for Schr\"odinger operators on bounded Lipschitz domains in $\bbR^n$, $n\ge 2$, with generalized Robin boundary conditions. 
\end{abstract}
%%%%%%%%%%%%%%%%%%%%%%%%%%%%%%%%%%%%%%%%
%%%%%%%%%%%%%%%%%%%%%%%%%%%%%%%%%%%%%%%%

\maketitle

%%%%%%%%%%%%%%%%%%%%%%%%%%%%%%%%%%%%%%%%
%%%%%%%%%%%%%%%%%%%%%%%%%%%%%%%%%%%%%%%%
\section{Introduction}\label{s1}
%%%%%%%%%%%%%%%%%%%%%%%%%%%%%%%%%%%%%%%%
%%%%%%%%%%%%%%%%%%%%%%%%%%%%%%%%%%%%%%%%

This paper is a direct continuation of our recent paper \cite{GM08} in which we studied Schr\"odinger operators on bounded Lipschitz and $C^{1,r}$-domains with generalized Robin boundary conditions and discussed associated Robin-to-Dirichlet maps and 
Krein-type resolvent formulas. The paper \cite{GM08}, in turn, was a continuation of the earlier papers \cite{GLMZ05} and \cite{GMZ07},  where we studied general, not necessarily self-adjoint, Schr\"odinger operators on $C^{1,r}$-domains $\Om\subset \bbR^n$, $n\in\bbN$, $n\ge 2$, with compact boundaries $\dOm$, $(1/2)<r<1$ (including unbounded domains, i.e., exterior domains) with Dirichlet and Neumann boundary conditions on $\dOm$. Our results also applied to convex domains $\Om$ and to domains satisfying a uniform exterior ball condition. In addition, a careful discussion of locally singular potentials $V$ with close to optimal local behavior of $V$ was provided in \cite{GLMZ05} and \cite{GMZ07}. 

In the current paper and in \cite{GM08}, we are exploring a different direction: Rather than discussing potentials with close to optimal local behavior, we will assume 
that $V \in L^\infty(\Om; d^n x)$ and hence essentially 
replace it by zero nearly everywhere in this paper. On the other hand, 
instead of treating Dirichlet and Neumann boundary conditions at $\dOm$, we 
now consider generalized Robin and again Dirichlet boundary conditions, but 
under minimal smoothness conditions on the domain $\Om$, that is, we now 
consider Lipschitz domains $\Om$. Additionally, to reduce some technicalities, 
we will assume that $\Om$ is bounded throughout this paper. The principal new result in this paper is a derivation of Krein-type resolvent formulas for Schr\"odinger operators on bounded Lipschitz domains $\Om$ in connection with two different generalized Robin boundary conditions on $\dOm$.

In Section \ref{s2} we recall our recent detailed discussion of self-adjoint Laplacians with generalized Robin (and Dirichlet) boundary conditions on $\dOm$ in \cite{GM08}. In Section \ref{s3} we summarize generalized Robin and Dirichlet boundary value problems and introduce associated Robin-to-Dirichlet and Dirichlet-to-Robin maps following 
\cite{GM08}. Section \ref{s4} is devoted to Krein-type resolvent formulas connecting Dirichlet and generalized Robin Laplacians with the help of the Robin-to-Dirichlet map. Section \ref{s5} contains our principal new results and studies Robin-to-Robin maps and general Krein-type formulas involving Robin-to-Robin maps. Appendix \ref{sA} collects useful material on Sobolev spaces and trace maps for Lipschitz domains. Appendix \ref{sB} summarizes pertinent facts on sesquilinear forms and their associated linear operators. 

While we formulate and prove all results in this paper for self-adjoint 
generalized Robin Laplacians and Dirichlet Laplacians, we emphasize that 
all results in this paper immediately extend to closed Schr\"odinger operators 
$H_{\Theta,\Om} = -\Delta_{\Theta,\Om} + V$, 
$\dom\big(H_{\Theta,\Om}\big) = \dom\big(-\Delta_{\Theta,\Om}\big)$  
in $\LOm$ for (not necessarily real-valued) potentials $V$ satisfying 
$V \in L^\infty(\Om; d^n x)$, by consistently replacing $-\Delta$ by 
$-\Delta + V$, etc. More generally, all results extend directly to 
Kato--Rellich bounded potentials $V$ relative to $-\Delta_{\Theta,\Om}$ 
with bound less than one.  

Next, we briefly list most of the notational conventions used throughout 
this paper. Let $\cH$ be a separable complex Hilbert space, 
$(\dott,\dott)_{\cH}$ the scalar product in $\cH$ 
(linear in the second factor), and $I_{\cH}$ the identity operator in $\cH$.
Next, let $T$ be a linear operator mapping (a subspace of) a
Banach space into another, with $\dom(T)$ and $\ran(T)$ denoting the
domain and range of $T$. The spectrum (resp., essential spectrum) of a 
closed linear operator in $\cH$ will be denoted by $\sigma(\dott)$ 
(resp., $\sigma_{\rm ess}(\dott)$). The Banach spaces of bounded and compact linear 
operators in $\cH$ are denoted by $\cB(\cH)$ and $\cB_\infty(\cH)$, 
respectively. Similarly, $\cB(\cH_1,\cH_2)$ and $\cB_\infty (\cH_1,\cH_2)$ 
will be used for bounded and compact operators between two Hilbert spaces 
$\cH_1$ and $\cH_2$. Moreover, $\cX_1 \hookrightarrow \cX_2$ denotes the 
continuous embedding of the Banach space $\cX_1$ into the Banach space 
$\cX_2$. Throughout this manuscript, if $X$ denotes a Banach space, $X^*$ 
denotes the {\it adjoint space} of continuous conjugate linear functionals 
on $X$, that is, the {\it conjugate dual space} of $X$ (rather than the usual 
dual space of continuous linear functionals 
on $X$). This avoids the well-known awkward distinction between adjoint 
operators in Banach and Hilbert spaces (cf., e.g., the pertinent discussion 
in \cite[p.\,3--4]{EE89}). 

Finally, a notational comment: For obvious reasons in connection 
with quantum mechanical applications, we will, with a slight abuse of 
notation, dub $-\Delta$ (rather than $\Delta$) as the ``Laplacian'' in 
this paper. 

%%%%%%%%%%%%%%%%%%%%%%%%%%%%%%%%%%%%%%%%
%%%%%%%%%%%%%%%%%%%%%%%%%%%%%%%%%%%%%%%%
\section{Laplace Operators with Generalized Robin Boundary Conditions}
\label{s2}
%%%%%%%%%%%%%%%%%%%%%%%%%%%%%%%%%%%%%%%%
%%%%%%%%%%%%%%%%%%%%%%%%%%%%%%%%%%%%%%%%

In this section we recall various properties of general 
Laplacians $-\Delta_{\Theta,\Om}$ in $L^2(\Om;d^n x)$ including Dirichlet,
$-\Delta_{D,\Om}$, and Neumann, $-\Delta_{N,\Om}$, Laplacians, generalized 
Robin-type Laplacians, and Laplacians corresponding to classical Robin 
boundary conditions associated with bounded open Lipschitz domains. For details we refer to our recent paper \cite{GM08}.

We start with introducing our precise assumptions on the set $\Omega$ and the 
boundary operator $\Theta$ which subsequently will be employed in defining 
the boundary condition on $\dOm$:

%%%%%%%%%%%%%%%%%%%%%%%%%%%%%%%%%%%%%%%
\begin{hypothesis} \lb{h2.1}
Let $n\in\bbN$, $n\geq 2$, and assume that $\Om\subset{\bbR}^n$ is
an open, bounded, nonempty Lipschitz domain.
% with a compact, nonempty boundary $\dOm$. 
\end{hypothesis}
%%%%%%%%%%%%%%%%%%%%%%%%%%%%%%%%%%%%%%%

We refer to Appendix \ref{sA} for more details on Lipschitz domains.

For simplicity of notation we will denote the identity operators in $\LOm$ and 
$\LdOm$ by $I_{\Om}$ and $I_{\dOm}$, respectively.
In addition, we refer to Appendix \ref{sA} for our notation in
connection with Sobolev spaces.
 
%%%%%%%%%%%%%%%%%%%%%%%%%%%%%%%%%%%%%%%
\begin{hypothesis} \lb{h2.2}
Assume Hypothesis \ref{h2.1} and suppose that $a_{\Theta}$ is a closed 
sesquilinear form in $\LdOm$ with domain $H^{1/2}(\dOm)\times H^{1/2}(\dOm)$, 
bounded from below by $c_{\Theta}\in\bbR$ $($hence, in particular, $a_{\Theta}$
is symmetric$)$. Denote by $\Theta \ge c_{\Theta}I_{\dOm}$ the self-adjoint 
operator in $\LdOm$ 
uniquely associated with $a_{\Theta}$ $($cf.\ \eqref{B.25}$)$ and by 
$\wti \Theta \in \cB\big(H^{1/2}(\dOm),H^{-1/2}(\dOm)\big)$ the extension 
of $\Theta$ as discussed in \eqref{B.24a} and \eqref{B.28a}.
\end{hypothesis}
%%%%%%%%%%%%%%%%%%%%%%%%%%%%%%%%%%%%%%%

Thus one has  
\begin{align}
& \big\langle f,\wti\Theta\,g\big\rangle_{1/2} 
= \ol{\big\langle g, \wti \Theta \, f \big\rangle_{1/2}}, 
 \quad f, g \in H^{1/2}(\dOm).     \lb{2.1}  \\
& \big\langle f, \wti \Theta \, f \big\rangle_{1/2} 
\geq c_{\Theta}\|f\|_{L^2(\dOm;d^{n-1}\omega)}^2,  
\quad f \in H^{1/2}(\dOm).  \lb{2.2} 
\end{align}

Here the sesquilinear form 
\begin{equation}
\langle \dott, \dott \rangle_{s}={}_{H^{s}(\dOm)}\langle\dott,\dott 
\rangle_{H^{-s}(\dOm)}\colon H^{s}(\dOm)\times H^{-s}(\dOm) 
\to \bbC, \quad s\in [0,1],   
\end{equation}
(antilinear in the first, linear in the second factor), denotes the duality 
pairing between $H^s(\dOm)$ and  
\begin{equation}
H^{-s}(\dOm) = \big(H^s(\dOm)\big)^*,  \quad s\in [0,1],    \lb{2.3}
\end{equation}
such that
\begin{equation}
\langle f, g \rangle_s = \int_{\dOm} d^{n-1} \omega(\xi) \, \ol{f(\xi)} 
g(\xi), \quad 
f \in H^s(\dOm), \; g \in L^2(\dOm; d^{n-1} \omega) \hookrightarrow 
H^{-s}(\dOm),  
\; s\in [0,1],  \lb{2.4}
\end{equation}
and $d^{n-1} \omega$ denotes the surface measure on $\dOm$. 

Hypothesis \ref{h2.1} on $\Om$ is used throughout this paper. 
Similarly, Hypothesis \ref{h2.2} is assumed whenever the boundary 
operator $\wti \Theta$ is involved. (Later in this section, and the next, 
we will occasionally strengthen our hypotheses.) 

We introduce the boundary trace operator $\ga_D^0$ (the Dirichlet trace) by
\begin{equation}
\ga_D^0\colon C(\ol{\Om})\to C(\dOm), \quad \ga_D^0 u = u|_\dOm.   \lb{2.5}
\end{equation}
Then there exists a bounded, linear operator $\gamma_D$ (cf., e.g., 
\cite[Theorem 3.38]{Mc00}),
\begin{align}
\begin{split}
& \ga_D\colon H^{s}(\Om)\to H^{s-(1/2)}(\dOm) \hookrightarrow \LdOm,
\quad 1/2<s<3/2, \lb{2.6}  \\
& \ga_D\colon H^{3/2}(\Om)\to H^{1-\varepsilon}(\dOm) \hookrightarrow \LdOm,
\quad \varepsilon \in (0,1), 
\end{split}
\end{align}
whose action is compatible with that of $\ga_D^0$. That is, the two
Dirichlet trace  operators coincide on the intersection of their
domains. Moreover, we recall that 
\begin{equation}
\ga_D\colon H^{s}(\Om)\to H^{s-(1/2)}(\dOm) \, \text{ is onto for  
$1/2<s<3/2$.} \lb{2.6a}
\end{equation}

While, in the class of bounded Lipschitz subdomains in $\bbR^n$, 
the end-point cases $s=1/2$ and $s=3/2$ of 
$\gamma_D\in\cB\bigl(H^{s}(\Om),H^{s-(1/2)}(\dOm)\bigr)$ fail, we nonetheless
have
\begin{eqnarray}\label{A.62x}
\ga_D\in \cB\big(H^{(3/2)+\eps}(\Om), H^{1}(\dOm)\big), \quad \eps>0.
\end{eqnarray}
See Lemma \ref{lA.6x} for a proof. Below we augment this with the following
result: 

%%%%%%%%%%%%%%%%%%%%
\begin{lemma}\label{Gam-L1}
Assume Hypothesis \ref{h2.1}. Then for each $s>-3/2$, the restriction 
to boundary operator \eqref{2.5} extends to a linear operator 
\begin{eqnarray}\label{Mam-1}
\gamma_D:\bigl\{u\in H^{1/2}(\Omega)\,\big|\,\Delta u\in H^{s}(\Omega)\bigr\}
\to L^2(\partial\Omega;d^{n-1}\omega),
\end{eqnarray}
is compatible with \eqref{2.6}, and is bounded when 
$\{u\in H^{1/2}(\Omega)\,|\,\Delta u\in H^{s}(\Omega)\bigr\}$ is equipped
with the natural graph norm $u\mapsto \|u\|_{H^{1/2}(\Omega)}
+\|\Delta u\|_{H^{s}(\Omega)}$. 

Furthermore, for each $s>-3/2$, the restriction 
to boundary operator \eqref{2.5} also extends to a linear operator 
\begin{eqnarray}\label{Mam-2}
\gamma_D:\bigl\{u\in H^{3/2}(\Omega)\,\big|\,\Delta u\in H^{1+s}(\Omega)\bigr\}
\to H^1(\partial\Omega), 
\end{eqnarray}
which, once again, is compatible with \eqref{2.6}, and is bounded when 
$\{u\in H^{3/2}(\Omega)\,|\,\Delta u\in H^{1+s}(\Omega)\bigr\}$ is equipped
with the natural graph norm $u\mapsto \|u\|_{H^{3/2}(\Omega)}
+\|\Delta u\|_{H^{1+s}(\Omega)}$. 
\end{lemma}
%%%%%%%%%%%%%%%%%%%%

Next, we introduce the operator $\ga_N$ (the Neumann trace) by
\begin{align}
\ga_N = \nu\cdot\ga_D\nabla \colon H^{s+1}(\Om)\to \LdOm, \quad
1/2<s<3/2, \lb{2.7} 
\end{align}
where $\nu$ denotes the outward pointing normal unit vector to
$\partial\Om$. It follows from \eqref{2.6} that $\ga_N$ is also a
bounded operator. We wish to further extend the action of the Neumann trace
operator \eqref{2.7} to other (related) settings. To set the stage, 
assume Hypothesis~\ref{h2.1} and recall that the inclusion 
\begin{equation}\lb{inc-1}
\iota:H^s(\Omega)\hookrightarrow \bigl(H^1(\Omega)\bigr)^*,\quad  s>-1/2,
\end{equation}
is well-defined and bounded. Then, we introduce the weak Neumann trace operator 
\begin{equation}\lb{2.8}
\wti\ga_N\colon\big\{u\in H^1(\Om)\,\big|\,\Delta u\in H^s(\Om)\big\} 
\to H^{-1/2}(\dOm),\quad  s>-1/2,  
\end{equation}
as follows: Given $u\in H^1(\Om)$ with $\Delta u \in H^s(\Om)$ 
for some $s>-1/2$, we set (with $\iota$ as in \eqref{inc-1})
\begin{align} \lb{2.9}
\langle \phi, \wti\ga_N u \rangle_{1/2}
=\int_\Om d^n x\,\ol{\nabla \Phi(x)} \cdot \nabla u(x)  
+ {}_{H^1(\Om)}\langle \Phi, \iota(\Delta u)\rangle_{(H^1(\Om))^*}, 
\end{align}
for all $\phi\in H^{1/2}(\dOm)$ and $\Phi\in H^1(\Om)$ such that
$\ga_D\Phi = \phi$. We note that this definition is
independent of the particular extension $\Phi$ of $\phi$, and that
$\wti\ga_N$ is a bounded extension of the Neumann trace operator
$\ga_N$ defined in \eqref{2.7}. As was the case of the Dirichlet trace, 
the (weak) Neumann trace operator \eqref{2.8}, \eqref{2.9} is onto  
(cf.\ \cite{GM08}). For additional details we refer to equations 
\eqref{A.11}--\eqref{A.16}. Next, we wish to discuss the end-point 
case $s=1/2$ of \eqref{2.7}. 
%%%%%%%%%%%%%%%%%%%%%%%%%%
\begin{lemma}\label{Neu-tr}
Assume Hypothesis \ref{h2.1}. 
Then the Neumann trace operator \eqref{2.7} also extends to
\begin{eqnarray}\label{MaX-1}
\wti\gamma_N:\bigl\{u\in H^{3/2}(\Omega)\,\big|\,\Delta u\in L^2(\Omega;d^nx)\bigr\}
\to L^2(\partial\Omega;d^{n-1}\omega) 
\end{eqnarray}
in a bounded fashion when the space 
$\{u\in H^{3/2}(\Omega)\,|\,\Delta u\in L^2(\Omega;d^nx)\bigr\}$ is equipped
with the natural graph norm $u\mapsto \|u\|_{H^{3/2}(\Omega)}
+\|\Delta u\|_{L^2(\Omega;d^nx)}$. This extension is compatible 
with \eqref{2.8}.
\end{lemma}
%%%%%%%%%%%%%%%%%%%%%%%%%%

For future purposes, we shall need yet another extension of the concept of 
Neumann trace. This requires some preparations 
(throughout, Hypothesis~\ref{h2.1} is enforced). First, we recall that, 
as is well-known (see, e.g., \cite{JK95}), one has the natural identification 
\begin{eqnarray}\lb{jk-9}
\big(H^{1}(\Om)\big)^*\equiv
\big\{u\in H^{-1}(\bbR^n)\,\big|\,\supp\, (u) \subseteq\overline{\Omega}\big\}.
\end{eqnarray} 
Note that the latter is a closed subspace of $H^{-1}(\bbR^n)$. 
In particular, if $R_{\Omega}u=u|_{\Omega}$ denotes the operator 
of restriction to $\Omega$ (considered in the sense of distributions), then 
\begin{eqnarray}\lb{jk-10}
R_{\Omega}:\big(H^{1}(\Om)\big)^*\to  H^{-1}(\Om)
\end{eqnarray}
is well-defined, linear and bounded. Furthermore, the 
composition of $R_\Omega$ in \eqref{jk-10} with $\iota$ in \eqref{inc-1}
is the natural inclusion of $H^s(\Om)$ into $H^{-1}(\Om)$. 
Next, given $z\in\bbC$, set 
\begin{eqnarray}\lb{2.88X}
W_z(\Om)=\bigl\{(u,f)\in H^1(\Om)\times\bigl(H^1(\Om)\bigr)^*\,\big|\,
(-\Delta-z)u=f|_{\Omega}\mbox{ in }\mathcal{D}'(\Om)\bigr\}, 
\end{eqnarray}
equipped with the norm inherited from 
$H^1(\Om)\times\bigl(H^1(\Om)\bigr)^*$. We then denote by
\begin{equation}\lb{2.8X}
\wti\ga_{\cN}\colon W_z(\Om)\to  H^{-1/2}(\dOm)  
\end{equation}
the {\it ultra weak} Neumann trace operator defined by
\begin{align}\lb{2.9X}
\begin{split}
\langle\phi,\wti\ga_{\cN} (u,f)\rangle_{1/2}
&= \int_\Om d^n x\,\ol{\nabla \Phi(x)} \cdot \nabla u(x)  \\
& \quad -z\,\int_\Om d^n x\,\ol{\Phi(x)}u(x)  
-{}_{H^1(\Om)}\langle \Phi, f\rangle_{(H^1(\Om))^*}, \quad 
(u,f)\in W_z(\Om), 
\end{split}
\end{align}
for all $\phi\in H^{1/2}(\dOm)$ and $\Phi\in H^1(\Om)$ such that 
$\ga_D\Phi=\phi$. Once again, this definition is independent of the 
particular extension $\Phi$ of $\phi$. Also, as was the case of the Dirichlet 
trace, the ultra weak Neumann trace operator \eqref{2.8X}, \eqref{2.9X} 
is onto (this is a corollary of Theorem \ref{t3.XV}). For additional details 
we refer to equations \eqref{A.11}--\eqref{A.16}. 

The relationship between the ultra weak Neumann trace operator 
\eqref{2.8X}, \eqref{2.9X} and the weak Neumann trace operator 
\eqref{2.8}, \eqref{2.9} can be described as follows. Given 
$s>-1/2$ and $z\in\bbC$, denote by 
\begin{eqnarray}\lb{2.10X}
j_z:\{u\in H^1(\Om)\,\big|\,\Delta u\in H^s(\Om)\big\} \to W_z(\Om)
\end{eqnarray}
the injection 
\begin{eqnarray}\lb{2.11X}
j_z(u)=(u,\iota(-\Delta u -zu)),\quad  u\in H^1(\Om), \; \Delta u\in H^s(\Om),
\end{eqnarray}
where $\iota$ is as in \eqref{inc-1}. Then 
\begin{eqnarray}\lb{2.12X}
\wti\gamma_{\cN}\circ j_z=\wti\ga_N.
\end{eqnarray}
Thus, from this perspective, $\wti\ga_{\cN}$ can also be regarded as a bounded 
extension of the Neumann trace operator $\ga_N$ defined in \eqref{2.7}.

Moving on, we shall now describe a family of self-adjoint Laplace operators 
$-\Delta_{\Theta,\Om}$ in $L^2(\Om; d^n x)$ indexed by the boundary 
operator $\Theta$. We will refer to $-\Delta_{\Theta,\Om}$ as the 
generalized Robin Laplacian. 

%%%%%%%%%%%%%%%%%%%%
\begin{theorem}  \lb{t2.3}
Assume Hypothesis \ref{h2.2}. Then the generalized Robin Laplacian, 
 $-\Delta_{\Theta,\Om}$, defined by 
\begin{equation}
-\Delta_{\Theta,\Om} = -\Delta, \quad \dom(-\Delta_{\Theta,\Om}) = 
\big\{u\in H^1(\Om)\,\big|\, \Delta u \in L^2(\Om;d^nx); \, 
\big(\wti\gamma_N + \wti \Theta \gamma_D\big) u =0 
\text{ in $H^{-1/2}(\dOm)$}\big\},  
\lb{2.20}
\end{equation}
is self-adjoint and bounded from below in $L^2(\Om;d^nx)$. Moreover,
\begin{equation}
\dom\big(|-\Delta_{\Theta,\Om}|^{1/2}\big) = H^1(\Om).   \lb{2.21}
\end{equation}
In addition,
$-\Delta_{\Theta,\Om}$, has purely discrete spectrum bounded from below, 
in particular, 
\begin{equation}
\sigma_{\rm ess}(-\Delta_{\Theta,\Om}) = \emptyset.  
\lb{2.21a}
\end{equation}
\end{theorem}
%%%%%%%%%%%%%%%%%%%%

The important special case where $\Theta$ corresponds to the operator of 
multiplication by a real-valued, essentially bounded function $\theta$ 
leads to Robin boundary conditions we discuss next:

%%%%%%%%%%%%%%%%%%%%
\begin{corollary} \lb{c2.3a}
In addition to Hypothesis \ref{h2.1}, assume that $\Theta$ is the operator 
of multiplication in $L^2(\dOm; d^{n-1} \omega)$ by the real-valued 
function $\theta$ satisfying $\theta \in L^\infty(\dOm; d^{n-1} \omega)$. 
Then $\Theta$ satisfies the conditions in Hypothesis \ref{h2.2}
resulting in the self-adjoint and bounded from below Laplacian 
$-\Delta_{\theta,\Om}$ in $L^2(\Om; d^n x)$ 
with Robin boundary conditions on $\dOm$ in \eqref{2.20} given by
\begin{equation}
(\wti\gamma_N + \theta \gamma_D) u= 0\,\text{ in } \, H^{-1/2}(\dOm).   
\lb{2.38a}
\end{equation}
\end{corollary}
%%%%%%%%%%%%%%%%%%%%

%%%%%%%%%%%%%%%%%%%%
\begin{remark}\lb{r2.4}
$(i)$ In the case of a smooth boundary $\partial\Om$, the boundary 
conditions in \eqref{2.38a} are also called ``classical'' boundary 
conditions (cf., e.g., \cite{Si78}); in the more general case of bounded 
Lipschitz domains we also refer to \cite{AW03} and \cite[Ch.\ 4]{Wa02} 
in this context. Next, we point out that, in \cite{LaSh}, the authors 
have dealt with the case of Laplace operators in bounded Lipschitz domains, 
equipped with local boundary conditions of Robin-type, with 
boundary data in $L^p(\dOm;d^{n-1}\omega)$, and produced nontangential 
maximal function estimates. For the case $p=2$, when our setting agrees
with that of \cite{LaSh}, some of our results in this section and the 
following are a refinement of those in \cite{LaSh}. Maximal $L^p$-regularity 
and analytic contraction semigroups of Dirichlet and Neumann Laplacians on 
bounded Lipschitz domains were studied in \cite{Wo07}. Holomorphic 
$C_0$-semigroups of the Laplacian with Robin boundary conditions on 
bounded Lipschitz domains have been discussed in \cite{Wa06}. 
Moreover, Robin boundary conditions for elliptic boundary value problems 
on arbitrary open domains were first studied by Maz'ya \cite{Ma81}, 
\cite[Sect.\ 4.11.6]{Ma85}, and subsequently in \cite{DD97} 
(see also \cite{Da00} which treats the case of the Laplacian). 
In addition, Robin-type boundary conditions involving measures on the 
boundary for very general domains $\Omega$ were intensively discussed 
in terms of quadratic forms and capacity methods in the literature, and 
we refer, for instance, to \cite{AW03}, \cite{AW03a}, \cite{BW06}, 
\cite{Wa02}, and the references therein.  \\
$(ii)$ In the special case $\theta=0$ (resp., $\wti \Theta =0$), that is, 
in the case of the Neumann Laplacian, we will also use the notation 
\begin{equation}
-\Delta_{N,\Om} = -\Delta_{0,\Om}.     \lb{2.38b}
\end{equation}
\end{remark}
%%%%%%%%%%%%%%%%%%%%

The case of the Dirichlet Laplacian $-\Delta_{D,\Om}$ associated with 
$\Om$ formally corresponds to $\Theta =\infty$ and so we recall its treatment  
in the next result. To state it, recall that, given a bounded Lipschitz 
domain $\Omega\subset\bbR^n$, 
\begin{equation}\label{H-zer}
H_0^1(\Om)=\{u\in H^1(\Om)\,|\,\gamma_Du=0\mbox{ on }\partial\Omega\}.
\end{equation}

%%%%%%%%%%%%%%%%%%%%
\begin{theorem}  \lb{t2.5}
Assume Hypothesis \ref{h2.1}. Then the Dirichlet Laplacian, 
$-\Delta_{D,\Om}$, defined by 
\begin{align}
-\Delta_{D,\Om} = -\Delta, \quad \dom(-\Delta_{D,\Om}) &= 
\big\{u\in H^1(\Om)\,\big|\, \Delta u \in L^2(\Om;d^n x); \, 
\gamma_D u =0 \text{ in $H^{1/2}(\dOm)$}\big\}   \no \\
&= \big\{u\in H_0^1(\Om)\,\big|\, \Delta u \in L^2(\Om;d^n x)\big\},  \lb{2.39}
\end{align}
is self-adjoint and strictly positive in $L^2(\Om;d^nx)$. Moreover,
\begin{equation}
\dom\big((-\Delta_{D,\Om})^{1/2}\big) = H^1_0(\Om).   \lb{2.40}
\end{equation}
\end{theorem}
%%%%%%%%%%%%%%%%%%%%

Since $\Om$ is open and bounded, it is well-known that $-\Delta_{D,\Om}$ has 
purely discrete spectrum contained in $(0,\infty)$, in particular, 
$\sigma_{\rm ess}(-\Delta_{D,\Om})=\emptyset$.

%%%%%%%%%%%%%%%%%%%%%%%%%%%%%%%%%%%%%%%
%%%%%%%%%%%%%%%%%%%%%%%%%%%%%%%%%%%%%%%
\section{Generalized Robin and Dirichlet Boundary Value Problems \\
and Robin-to-Dirichlet and Dirichlet-to-Robin Maps} \label{s3}
%%%%%%%%%%%%%%%%%%%%%%%%%%%%%%%%%%%%%%%
%%%%%%%%%%%%%%%%%%%%%%%%%%%%%%%%%%%%%%%

This section is devoted to generalized Robin and Dirichlet boundary value problems associated with the Helmholtz differential expression $-\Delta - z$ in connection with the open set $\Omega$. In addition, we provide a detailed discussion of Robin-to-Dirichlet maps, $M_{\Theta,D,\Om}^{(0)}$,  in $\LdOm$. Again, the material in this section is taken from \cite[Sect.\ 3]{GM08}.

In this section we strengthen Hypothesis \ref{h2.2} by adding assumption 
\eqref{3.5} below: 

%%%%%%%%%%%%%%%%%%%%%%%%%%%%%%%%%%%%%%%
\begin{hypothesis} \lb{h3.1}
In addition to Hypothesis \ref{h2.2} suppose that 
\begin{equation}\lb{3.5}
\wti \Theta \in \cB_{\infty}\big(H^1(\dOm),L^2(\dOm;d^{n-1} \omega)\big).   
\end{equation}
\end{hypothesis}
%%%%%%%%%%%%%%%%%%%%%%%%%%%%%%%%%%%%%%%

\noindent We note that \eqref{3.5} is satisfied whenever there 
exists some $\varepsilon>0$ such that
\begin{equation}\lb{3.5bis}
\wti \Theta \in \cB\big(H^{1-\varepsilon}(\dOm),L^2(\dOm;d^{n-1} \omega)\big). 
\end{equation}

We recall the definition of the weak Neumann trace operator $ \wti\ga_N$ in 
\eqref{2.8}, \eqref{2.9} and start with the Helmholtz Robin boundary value 
problems:

%%%%%%%%%%%%%%%%%%%%%%%%%%%%%%%%%%%%%%%
\begin{theorem} \lb{t3.2} 
Assume Hypothesis \ref{h3.1} and suppose that 
$z\in\bbC\backslash\si(-\Delta_{\Theta,\Om})$. Then for every $g\in\LdOm$, 
the following generalized Robin boundary value problem,
\begin{equation} \lb{3.6}
\begin{cases}
(-\Delta - z)u = 0 \text{ in }\,\Om,\quad u \in H^{3/2}(\Om), \\
\big(\wti\ga_N + \wti \Theta \gamma_D\big) u = g \text{ on } \,\dOm,
\end{cases}
\end{equation}
has a unique solution  $u=u_\Theta$. This solution $u_\Theta$ satisfies
\begin{eqnarray}\lb{3.6a}
\begin{array}{l}
\ga_D u_\Theta \in H^1(\dOm), \quad \wti \ga_N u_\Theta 
\in L^2(\dOm;d^{n-1}\omega),
\\[4pt]
\|\ga_D u_\Theta\|_{H^1(\dOm)}+\|\wti\ga_N u_\Theta\|_{L^2(\dOm;d^{n-1}\omega)},
\leq C\|g\|_{\LdOm}
\end{array}
\end{eqnarray}
and
\begin{equation}
\|u_\Theta\|_{H^{3/2}(\Omega)} \leq C\|g\|_{\LdOm}, \lb{3.7}
\end{equation}
for some constant constant $C= C(\Theta,\Omega,z)>0$. Finally,    
\begin{equation}
\big[\ga_D (-\Delta_{\Theta,\Om}-\ol{z}I_\Om)^{-1}\big]^* \in
\cB\big(\LdOm, H^{3/2}(\Om)\big),   \lb{3.8}
\end{equation}
and the solution $u_\Theta$ is given by the formula  
\begin{equation}
u_\Theta = \big(\ga_D (-\Delta_{\Theta,\Om}-\ol{z}I_\Om)^{-1}\big)^*g. 
\lb{3.9}
\end{equation}
\end{theorem}
%%%%%%%%%%%%%%%%%%%%%%%%%%%%%%%%%%%%%%%
 
Next, we turn to the Dirichlet case originally treated in 
\cite[Theorem 3.1]{GMZ07} but under stronger regularity conditions 
on $\Omega$. 

%%%%%%%%%%%%%%%%%%%%%%%%%%%%%%%%%
\begin{theorem} \lb{t3.3}
Assume Hypothesis \ref{h2.1} and suppose that 
$z\in\bbC\backslash\si(-\Delta_{D,\Om})$. Then for every $f \in H^1(\dOm)$, 
the following Dirichlet boundary value problem,
\begin{equation} \lb{3.31}
\begin{cases}
(-\Delta - z)u = 0 \text{ in }\, \Om, \quad u \in H^{3/2}(\Om), \\
\ga_D u = f \text{ on }\, \dOm,
\end{cases}
\end{equation}
has a unique solution $u=u_D$. This solution $u_D$ satisfies 
\begin{equation}\lb{Hh.3} 
\wti\ga_N u_D \in \LdOm\quad \mbox{and}\quad 
\|\wti\ga_N u_D\|_{\LdOm}\leq C_D\|f\|_{H^1(\dOm)},
\end{equation}
for some constant $C_D=C_D(\Omega,z)>0$. Moreover, 
\begin{equation}
\|u_D\|_{H^{3/2}(\Omega)} \leq C_D \|f\|_{H^1(\partial\Omega)}.  \lb{3.33}
\end{equation} 
Finally,      
\begin{equation}
\big[\wti\ga_N (-\Delta_{D,\Om}-{\ol z}I_\Om)^{-1}\big]^* \in 
\cB\big(H^1(\dOm), H^{3/2}(\Om)\big),   \lb{3.34} 
\end{equation}
and the solution $u_D$ is given by the formula
\begin{equation}
u_D = -\big(\wti\ga_N (-\Delta_{D,\Om}-\ol{z}I_\Om)^{-1}\big)^*f. 
\lb{3.35}
\end{equation}
\end{theorem}
%%%%%%%%%%%%%%%%%%%%%

In addition to Theorem \ref{t3.3}, we recall the following result.

%%%%%%%%%%%%%%%%%%%
\begin{lemma}\label{L-Bbb}
Assume Hypothesis \ref{h2.1} and suppose that 
$z\in\bbC\backslash\si(-\Delta_{D,\Om})$. Then 
\begin{equation}\lb{3.23}
\wti\ga_N (-\Delta_{D,\Om} - zI_\Om)^{-1} \in \cB\big(\LOm,\LdOm\big),
\end{equation}
and 
\begin{equation}\lb{3.24}
\big[\wti\ga_N (-\Delta_{D,\Om} - zI_\Om)^{-1}\big]^* 
\in \cB\big(\LdOm,\LOm\big).  
\end{equation}
\end{lemma}
%%%%%%%%%%%%%%%%%%%

Assuming Hypothesis \ref{h3.1}, we now introduce the
Dirichlet-to-Robin map $M_{D,\Theta,\Om}^{(0)}(z)$ 
associated with $(-\Delta-z)$ on $\Om$, as follows,
\begin{align}
M_{D,\Theta,\Om}^{(0)}(z) \colon
\begin{cases}
H^1(\dOm) \to \LdOm,  \\
\hspace*{10mm} f \mapsto - \big(\wti\ga_N +\wti\Theta\gamma_D\big)u_D,
\end{cases}  \quad z\in\bbC\backslash\si(-\Delta_{D,\Om}), \lb{3.44}
\end{align}
where $u_D$ is the unique solution of
\begin{align}
(-\Delta-z)u = 0 \,\text{ in }\Om, \quad u \in
H^{3/2}(\Om), \quad \ga_D u = f \,\text{ on }\dOm.   \lb{3.45}
\end{align}

Continuing to assume Hypothesis \ref{h3.1}, we next introduce the 
Robin-to-Dirichlet map $M_{\Theta,D,\Om}^{(0)}(z)$ associated with 
$(-\Delta-z)$ on $\Om$, as follows,
\begin{align}
M_{\Theta,D,\Om}^{(0)}(z) \colon \begin{cases} \LdOm \to H^1(\dOm),
\\
\hspace*{20.8mm} g \mapsto \ga_D u_{\Theta}, \end{cases}  \quad
z\in\bbC\backslash\si(-\Delta_{\Theta,\Om}), \lb{3.48}
\end{align}
where $u_{\Theta}$ is the unique solution of
\begin{align}
(-\Delta-z)u = 0 \,\text{ in }\Om, \quad u \in
H^{3/2}(\Om), \quad \big(\wti\ga_N + \wti \Theta\gamma_D\big)u = g 
\,\text{ on }\dOm.  \lb{3.49}
\end{align}
We note that Robin-to-Dirichlet maps have also been studied in \cite{Au04}.

Next we recall one of the main results in \cite{GM08}:
 
%%%%%%%%%%%%%%
\begin{theorem} \lb{t3.5} 
Assume Hypothesis \ref{h3.1}. Then 
\begin{equation}
M_{D,\Theta,\Om}^{(0)}(z) \in \cB\big(H^1(\partial\Om), \LdOm \big), \quad
z\in\bbC\backslash\si(-\Delta_{D,\Om}),   \lb{3.46}
\end{equation}
and 
\begin{equation}
M_{D,\Theta,\Om}^{(0)}(z) = \big(\wti\gamma_N+\wti\Theta\gamma_D\big)\big[
\big(\wti\gamma_N+\wti\Theta\gamma_D\big)
(-\Delta_{D,\Om} - \ol{z}I_\Om)^{-1}\big]^*, 
\quad z\in\bbC\backslash\si(-\Delta_{D,\Om}). \lb{3.47}
\end{equation}
Moreover, 
\begin{equation}
M_{\Theta,D,\Om}^{(0)}(z) \in \cB\big(\LdOm, H^1(\partial\Om) \big), \quad 
z\in\bbC\backslash\si(-\Delta_{\Theta,\Om}),    \lb{3.50}
\end{equation}
and, in fact,  
\begin{equation}
M_{\Theta,D,\Om}^{(0)}(z) \in \cB_\infty\big(\LdOm\big), \quad 
z\in\bbC\backslash\si(-\Delta_{\Theta,\Om}).  \lb{3.51}
\end{equation}
In addition, 
\begin{equation}
M_{\Theta,D,\Om}^{(0)}(z) = \gamma_D\big[\gamma_D
(-\Delta_{\Theta,\Om} - \ol{z}I_\Om)^{-1}\big]^*, \quad
z\in\bbC\backslash\si(-\Delta_{\Theta,\Om}). \lb{3.52}
\end{equation} 
Finally, let 
$z\in\bbC\backslash(\si(-\Delta_{D,\Om})\cup\si(-\Delta_{\Theta,\Om}))$. Then
\begin{equation}
M_{\Theta,D,\Om}^{(0)}(z) = - M_{D,\Theta,\Om}^{(0)}(z)^{-1}.   \lb{3.53}  
\end{equation}
\end{theorem}
%%%%%%%%%%%%%%

%%%%%%%%%%%%%%
\begin{remark} \lb{r3.6X}
In the above considerations, the special case $\Theta =0$ represents 
the frequently studied Neumann-to-Dirichlet and Dirichlet-to-Neumann maps
$M_{N,D,\Om}^{(0)}(z)$ and $M_{D,N,\Om}^{(0)}(z)$, respectively. That is, 
$M_{N,D,\Om}^{(0)}(z)=M_{0,D,\Om}^{(0)}(z)$ and 
$M_{D,N,\Om}^{(0)}(z)=M_{D,0,\Om}^{(0)}(z)$. Thus, as a corollary 
of Theorem \ref{t3.5} we have 
\begin{equation}
M_{N,D,\Om}^{(0)}(z) = - M_{D,N,\Om}^{(0)}(z)^{-1},   \lb{3.53X}  
\end{equation}
whenever Hypothesis \ref{h2.1} holds and 
$z\in\bbC\backslash(\si(-\Delta_{D,\Om})\cup\si(-\Delta_{N,\Om}))$. 
\end{remark}
%%%%%%%%%%%%%%%%%%%%%%%%%%%

%%%%%%%%%%%%%%
\begin{remark}\lb{r3.6}
We emphasize again that all results in this section immediately extend to  
Schr\"odinger operators $H_{\Theta,\Om} = -\Delta_{\Theta,\Om} + V$, 
$\dom\big(H_{\Theta,\Om}\big) = \dom\big(-\Delta_{\Theta,\Om}\big)$ in 
$\LOm$ for (not necessarily real-valued) potentials $V$ satisfying 
$V \in L^\infty(\Om; d^n x)$, or more generally, for potentials $V$ which 
are Kato--Rellich bounded with respect to $-\Delta_{\Theta,\Om}$ with bound 
less than one. Denoting the corresponding $M$-operators by $M_{D,N,\Om}(z)$ 
and $M_{\Theta,D,\Om}(z)$, respectively, we note, in particular, that 
\eqref{3.44}--\eqref{3.53} extend replacing $-\Delta$ by $-\Delta + V$ 
and restricting $z\in\bbC$ appropriately.
\end{remark}
%%%%%%%%%%%%%%

Our discussion of Weyl--Titchmarsh operators follows the earlier papers  
\cite{GLMZ05} and \cite{GMZ07}. For related literature on Weyl--Titchmarsh 
operators, relevant in the context
of boundary value spaces (boundary triples, etc.), we refer, for
instance, to \cite{ABMN05}, \cite{AP04}, \cite{BL07}, \cite{BMN06},
\cite{BMN00}--\cite{BMNW08}, \cite{DHMS00}-- \cite{DM95}, \cite{GKMT01}, \cite{GM09}, \cite[Ch.\ 3]{GG91}, \cite[Ch.\ 13]{Gr09}, \cite{MM06}, \cite{Ma04}, 
\cite{MPP07}, \cite{Pa87}, \cite{Pa02}, \cite{Po04}--\cite{Ry08a}, \cite{TS77}.

%%%%%%%%%%%%%%%%%%%%%%%%%%%%%%%%%%%%%%%
%%%%%%%%%%%%%%%%%%%%%%%%%%%%%%%%%%%%%%%
\section{Some Variants of Krein's Resolvent Formula Involving 
Robin-to-Dirichlet Maps} \label{s4}
%%%%%%%%%%%%%%%%%%%%%%%%%%%%%%%%%%%%%%%
%%%%%%%%%%%%%%%%%%%%%%%%%%%%%%%%%%%%%%%

In this section we recall some of the principal new results in \cite{GM08}, 
viz., variants of Krein's formula for the difference of resolvents of generalized Robin 
Laplacians and Dirichlet Laplacians on bounded Lipschitz domains. For details on the material in this section we refer to \cite{GM08}.

We start by weakening Hypothesis \ref{h3.1} by using assumption \eqref{3.5bisa} below: 

%%%%%%%%%%%%%%%%%%%%%%%%%
\begin{hypothesis} \lb{h3.1bis}
In addition to Hypothesis \ref{h2.2} suppose that 
\begin{equation}\lb{3.5bisa}
\wti\Theta\in\cB_\infty\big(H^{1/2}(\dOm),H^{-1/2}(\partial\Omega)\big). 
\end{equation}
\end{hypothesis}
%%%%%%%%%%%%%%%%%%%%%%%%%

We note that condition \eqref{3.5bisa} is satisfied 
if there exists some $\varepsilon>0$ such that 
\begin{equation}\lb{3.5D}
\wti \Theta \in \cB\big(H^{1/2-\varepsilon}(\dOm),H^{-1/2}(\partial\Omega)\big). 
\end{equation}

We wish to point out that Hypothesis \ref{h3.1} is indeed stronger 
than Hypothesis \ref{h3.1bis} since \eqref{3.5} implies, via duality and
interpolation (cf.\ the discussion in \cite{GM08}), that 
\begin{equation}\lb{4.3}
\wti\Theta  \in\cB_{\infty}\big(H^{s}(\dOm),H^{s-1}(\partial\Omega)\big),
\quad 0\leq s\leq 1.
\end{equation}

In our next two results below (Theorems \ref{t3.2bis}--\ref{t3.XV}) we 
discuss the solvability of the Dirichlet and Robin boundary value problems
with solution in the energy space $H^1(\Omega)$.  

%%%%%%%%%%%%%%%%%%%%%%%%%
\begin{theorem} \lb{t3.2bis} 
Assume Hypothesis \ref{h3.1bis} and suppose that 
$z\in\bbC\backslash\si(-\Delta_{\Theta,\Om})$. Then for every 
$g\in H^{-1/2}(\partial\Omega)$, 
the following generalized Robin boundary value problem,
\begin{equation} \lb{3.6bis}
\begin{cases}
(-\Delta - z)u = 0 \text{ in }\,\Om,\quad u \in H^{1}(\Om), \\
\big(\wti\ga_N + \wti \Theta \gamma_D\big) u = g \text{ on } \,\dOm,
\end{cases}
\end{equation}
has a unique solution $u=u_\Theta$. Moreover, there exists a constant 
$C= C(\Theta,\Omega,z)>0$ such that
\begin{equation}\lb{3.7bis}
\|u_\Theta\|_{H^{1}(\Omega)} \leq C\|g\|_{H^{-1/2}(\partial\Omega)}.  
\end{equation}
In particular, 
\begin{equation}\lb{3.8bis}
\big[\ga_D (-\Delta_{\Theta,\Om}-\ol{z}I_\Om)^{-1}\big]^* \in
\cB\big(H^{-1/2}(\partial\Omega), H^{1}(\Om)\big),  
\end{equation}
and the solution $u_\Theta$ of \eqref{3.6bis} is once again given by formula  
\eqref{3.9}. 
\end{theorem}
%%%%%%%%%%%%%%%%%%%%

%%%%%%%%%%%%%%%%%%%%
\begin{remark} \lb{r3.6bis}
As a byproduct of Theorem \ref{t3.2bis} (with $\Theta=0$) we obtain that the
weak Neumann trace $\widetilde\gamma_N$ in \eqref{2.8}, \eqref{2.9} is onto. 
\end{remark}
%%%%%%%%%%%%%%%%%%%%
In the following we denote by $\wti I_{\Om}$ the continuous inclusion 
(embedding) map of $H^1(\Omega)$ into $\bigl(H^1(\Omega)\bigr)^*$. 
By a slight abuse of notation, we also denote the continuous inclusion 
map of $H^1_0(\Omega)$ into $\bigl(H^1_0(\Omega)\bigr)^*$ by the same 
symbol $\wti I_{\Om}$. We recall the ultra weak Neumann trace operator 
$\wti\ga_{\cN}$ from \eqref{2.8X}, \eqref{2.9X}. Finally, assuming 
Hypothesis~\ref{h3.1bis}, we denote by 
\begin{equation} \lb{3.JqY}
- \wti \Delta_{\Theta,\Om}\in\cB\big(H^1(\Om),\big(H^1(\Om)\big)^*\big)
\end{equation}
the extension of $- \Delta_{\Theta,\Om}$ in accordance with \eqref{B.24a}. 
In particular, 
\begin{equation} \lb{3.JqZ}
{}_{H^1(\Om)}\langle u,- \wti \Delta_{\Theta,\Om}v\rangle_{(H^1(\Om))^*}
=\int_{\Om}d^nx\,\ol{\nabla u(x)}\cdot\nabla v(x)
+\big\langle \gamma_D u, \wti \Theta \gamma_D v \big\rangle_{1/2}, 
\quad  u,v\in H^1(\Om),
\end{equation}
and $- \Delta_{\Theta,\Om}$ is the restriction of 
$- \wti \Delta_{\Theta,\Om}$ to $L^2(\Om;d^nx)$ (cf.\ \eqref{B.25}). 
%%%%%%%%%%%%%%%%%%%%
\begin{theorem} \lb{t3.XV} 
Assume Hypothesis \ref{h3.1bis} and suppose that 
$z\in\bbC\backslash\si(-\Delta_{\Theta,\Om})$. Then for every 
$w\in (H^{1}(\Omega))^*$, the following generalized inhomogeneous Robin problem,
\begin{equation} \lb{3.Jq}
\begin{cases}
(-\Delta - z)u = w|_{\Om} \text{ in }\,{\mathcal{D}}'(\Om),
\quad u \in H^{1}(\Om), \\
\wti\ga_{\cN} (u,w)+ \wti \Theta \gamma_D u = 0 \text{ on } \,\dOm,
\end{cases}
\end{equation}
has a unique solution $u=u_{\Theta,w}$. Moreover, there exists a constant 
$C= C(\Theta,\Omega,z)>0$ such that
\begin{equation}\lb{3.Jq2}
\|u_{\Theta,w}\|_{H^{1}(\Omega)} \leq C\|w\|_{(H^{1}(\partial\Omega))^*}.  
\end{equation}
In particular, the operator $(-\Delta_{\Theta,\Om}-zI_\Om)^{-1}$,
$z\in\bbC\backslash\si(-\Delta_{\Theta,\Om})$, originally defined as a bounded 
operator on $\LOm$,  
\begin{align}\label{faH}
(-\Delta_{\Theta,\Om}-zI_\Om)^{-1} \in \cB\big(L^2(\Om;d^nx)\big),
\end{align}
can be extended to a mapping in $\cB\big(\big(H^{1}(\Om)\big)^*,H^1(\Om)\big)$, 
which in fact coincides with 
\begin{equation}\label{fcH}
\big(- \wti \Delta_{\Theta,\Om} -z \wti I_\Om\big)^{-1}
\in\cB\big(\big(H^{1}(\Om)\big)^*,H^1(\Om)\big).
\end{equation} 
\end{theorem}
%%%%%%%%%%%%%%%%%%%%%%
\begin{remark} \lb{r3.VF}
In the context of Theorem \ref{t3.XV}, it is useful to observe that 
for any $w\in \big(H^{1}(\Om)\big)^*$, the function 
$u=\big(- \wti \Delta_{\Theta,\Om} -z \wti I_\Om\big)^{-1}w\in H^1(\Om)$ 
satisfies 
\begin{equation}\label{fHy}
(-\Delta-z)u=w|_{\Omega}\, \text{ in }{\mathcal{D}}'(\Om),
\end{equation} 
where the restriction of $w$ to $\Omega$ is interpreted by 
regarding $w$ as a distribution in $H^{-1}(\Om)$ (cf. \eqref{jk-9}).
Indeed, the identification \eqref{jk-9} associates to a functional 
$w\in \big(H^{1}(\Om)\big)^*$ the distribution 
$\widehat{w}=w\circ R_{\Omega}\in H^{-1}(\bbR^n)=\bigl(H^1(\bbR^n)\bigr)^*$
(which happens to be supported in $\ol{\Om}$). Consequently, if 
for an arbitrary test function $\varphi\in C^\infty_0(\Om)$ we denote 
by $\wti\varphi\in C^\infty_0(\bbR^n)$ the extension of $\varphi$ by zero
outside $\Om$, we then have 
\begin{align}\label{fHy-2}
{}_{{\mathcal{D}}(\Om)}\langle\varphi,\widehat{w}|_{\Om}
\rangle_{{\mathcal{D}}'(\Om)}
&= {}_{{\mathcal{D}}(\bbR^n)}\langle\wti\varphi,
\widehat{w}\rangle_{{\mathcal{D}}'(\bbR^n)}
\nonumber\\ 
&= {}_{H^1(\bbR^n)}\langle\wti\varphi,
\widehat{w}\rangle_{(H^1(\bbR^n))^*}
={}_{H^1(\Om)}\langle R_{\Om}(\wti\varphi),w\rangle_{(H^1(\Om))^*}
\nonumber\\ 
&= {}_{H^1(\Om)}\langle\varphi,w\rangle_{(H^1(\Om))^*}
={}_{H^1(\Om)}\big\langle\varphi,
\big(-\wti\Delta_{\Theta,\Om}-z\wti I_\Om\big)u \big\rangle_{(H^1(\Om))^*}
\nonumber\\ 
&= {}_{H^1(\Om)}\big\langle\varphi,
\big(-\wti\Delta_{\Theta,\Om})u \big\rangle_{(H^1(\Om))^*}
-z\, (\varphi,u)_{L^2(\Om;d^nx)}
\nonumber\\ 
&= (\nabla\varphi,\nabla u)_{(L^2(\Om;d^nx))^n}
-z\, (\varphi,u)_{L^2(\Om;d^nx)}
\nonumber\\ 
&= ((-\Delta-\ol{z})\varphi,u)_{L^2(\Om;d^nx)},
\end{align}
on account of \eqref{3.JqZ}. This justifies \eqref{fHy}. 
\end{remark} 

%%%%%%%%%%%%%%%%%%%%%%%
\begin{remark} \lb{r3.YY}
Similar (yet simpler) considerations also show that the operator 
$(-\Delta_{D,\Om}-zI_\Om)^{-1}$, $z\in\bbC\backslash\si(-\Delta_{D,\Om})$, 
originally defined as bounded operator on $\LOm$,
\begin{equation}\label{fuTT}
(-\Delta_{D,\Om}-zI_\Om)^{-1} \in \cB\big(L^2(\Om;d^nx)\big),
\end{equation}
extends to a mapping
\begin{equation}\label{fcH3}
\big(-\wti \Delta_{D,\Om}-z \wti I_\Om\big)^{-1}
\in\cB\big(H^{-1}(\Om);H^1_0(\Om)\big).
\end{equation}
Here $- \wti \Delta_{D,\Om}\in\cB\big(H^1_0(\Om), H^{-1}(\Om)\big)$ is the 
extension of $- \Delta_{D,\Om}$ in accordance with \eqref{B.24a}. 
Indeed, the Lax--Milgram lemma applies and yields that 
\begin{equation}\label{fuRR}
\big(- \wti \Delta_{D,\Om}-z \wti I_\Om\big) \colon H^1_0(\Om)\to 
\big(H^1_0(\Om)\big)^* = H^{-1}(\Om) 
\end{equation}
is, in fact, an isomorphism whenever $z\in\bbC\backslash\si(-\Delta_{D,\Om})$.
\end{remark}
%%%%%%%%%%%%%%%%%%%%%%

%%%%%%%%%%%%%%%%%%%%%% 
\begin{corollary} \lb{t3.XU} 
Assume Hypothesis \ref{h3.1bis} and suppose that 
$z\in\bbC\backslash\si(-\Delta_{\Theta,\Om})$. Then the operator 
$M_{\Theta,D,\Om}^{(0)}(z)\in \cB\big(\LdOm\big)$ in \eqref{3.48}, \eqref{3.49} 
extends $($in a compatible manner\,$)$ to 
\begin{equation}\lb{3.50bis}
\wti M_{\Theta,D,\Om}^{(0)}(z) 
\in \cB\big(H^{-1/2}(\partial\Om) , H^{1/2}(\partial\Om) \big), \quad 
z\in\bbC\backslash\si(-\Delta_{\Theta,\Om}).    
\end{equation}
In addition, $\wti M_{\Theta,D,\Om}^{(0)}(z)$ permits the representation 
\begin{equation}
\wti M_{\Theta,D,\Om}^{(0)}(z) 
= \gamma_D \big(-\wti\Delta_{\Theta,\Om} - z \wti I_\Om\big)^{-1}\gamma_D^*, 
\quad z\in\bbC\backslash\si(-\Delta_{\Theta,\Om}). \lb{3.52bis}
\end{equation} 
The same applies to the adjoint $M_{\Theta,D,\Om}^{(0)}(z)^*
\in \cB\big(\LdOm\big)$ of $M_{\Theta,D,\Om}^{(0)}(z)$, resulting in the 
bounded extension $\big(\wti M_{\Theta,D,\Om}^{(0)}(z)\big)^* 
\in \cB\big(H^{-1/2}(\partial\Om)$, $H^{1/2}(\partial\Om) \big)$, 
$z\in\bbC\backslash\si(-\Delta_{\Theta,\Om})$.   
\end{corollary}
%%%%%%%%%%%%%%%%%%%%%%% 

%%%%%%%%%%%%%%%%%%%%%%%%
\begin{lemma} \lb{lA.3BB}
Assume Hypothesis \ref{h3.1bis} and suppose that 
$z\in\bbC\backslash(\si(-\Delta_{\Theta,\Om})\cup\si(-\Delta_{D,\Om}))$.
Then the following resolvent relation holds on $\bigl(H^1(\Omega)\bigr)^*$, 
\begin{align}\lb{Na1B}
\begin{split}
\big(-\wti\Delta_{\Theta,\Om}-z \wti I_\Om\big)^{-1} 
&=\big(-\wti\Delta_{D,\Om}-z \wti I_\Om\big)^{-1}\circ R_{\Omega}   \\ 
& \quad + \big(-\wti\Delta_{\Theta,\Om}-z \wti I_\Om\big)^{-1}\ga_D^*
\wti\gamma_{\cN}\big(\big(-\wti\Delta_{D,\Om}-z \wti I_\Om\big)^{-1}\circ R_{\Omega}
,I_{{\bbR}^n}\big).
\end{split}
\end{align}
\end{lemma}
%%%%%%%%%%%%%%%%%%%%%%%

We also recall the following regularity result for the Robin resolvent. 

%%%%%%%%%%%%%%%%%%%
\begin{lemma}\label{L-BbbZa}
Assume Hypothesis \ref{3.5} and suppose that 
$z\in\bbC\backslash\si(-\Delta_{\Theta,\Om})$. Then 
\begin{equation}\lb{Kr-a.1a}
(-\Delta_{\Theta,\Om} - zI_\Om)^{-1}:\LOm\to\bigl\{u\in H^{3/2}(\Omega)\,\big|\,
\Delta u\in L^2(\Omega;d^nx)\bigl\}
\end{equation}
is a well-defined bounded operator, where the space 
$\big\{u\in H^{3/2}(\Omega)\,\big|\,\Delta u\in L^2(\Omega;d^nx)\big\}$ is 
equipped with the natural graph norm 
$u\mapsto\|u\|_{H^{3/2}(\Omega)}+\|\Delta u\|_{L^2(\Omega;d^nx)}$. 
\end{lemma}
%%%%%%%%%%%%%%%%%%%

Under Hypothesis \ref{h3.1bis}, \eqref{fcH} and \eqref{2.6} yield 
\begin{eqnarray}\label{Obt.1}
\gamma_D\big(-\wti\Delta_{\Theta,\Omega}-z\wti I_{\Omega}\big)^{-1}
\in \cB\bigl((H^1(\Omega))^*,H^{1/2}(\partial\Omega)\bigr). 
\end{eqnarray}
Hence, by duality, 
\begin{eqnarray}\label{Obt.2}
\big[\gamma_D\big(-\wti\Delta_{\Theta,\Omega}-z\wti I_{\Omega}\big)^{-1}\big]^*
\in \cB\bigl(H^{-1/2}(\partial\Omega),H^1(\Omega)\bigr). 
\end{eqnarray}

Next we complement this with the following result. 

%%%%%%%%%%%%%%%%%%%%%%%%%%%%%%%%%%
\begin{corollary}\label{L-fR8}
Assume Hypothesis \ref{3.5} and suppose that 
$z\in\bbC\backslash\si(-\Delta_{\Theta,\Om})$. Then 
\begin{equation}\lb{Kr-a.2}
\gamma_D(-\Delta_{\Theta,\Om}-zI_\Om)^{-1}\in
\cB\bigl(\LOm,H^{1}(\partial\Omega)\bigr). 
\end{equation}
In particular, 
\begin{equation}\lb{Kr-a.3}
\big[\gamma_D(-\Delta_{\Theta,\Om}-zI_\Om)^{-1}\big]^*\in
\cB\bigl(H^{-1}(\partial\Omega),\LOm\bigr)\hookrightarrow
\cB\bigl(L^2(\partial\Omega;d^{n-1}\Omega),\LOm\bigr).
\end{equation}
In addition, the operator \eqref{Kr-a.3} is compatible with \eqref{Obt.2}
in the sense that 
\begin{eqnarray}\label{Obt.3}
\big[\gamma_D(-\Delta_{\Theta,\Omega}-zI_{\Omega})^{-1}\big]^*f
=\big[\gamma_D\big(-\wti\Delta_{\Theta,\Omega}-z\wti I_{\Omega}\big)^{-1}\big]^*f 
\mbox{ in } \LOm,  \; f\in H^{-1/2}(\partial\Omega).
\end{eqnarray}
As a consequence, 
\begin{eqnarray}\label{Obt.3bis}
\big[\gamma_D(-\Delta_{\Theta,\Omega}-zI_{\Omega})^{-1}\big]^*f
=\big[\gamma_D\big(-\wti\Delta_{\Theta,\Omega}-z \wti I_{\Omega}\big)^{-1}\big]^*f 
\mbox{ in }\LOm, \; f\in L^{2}(\partial\Omega;d^{n-1}\omega).
\end{eqnarray}
\end{corollary}
%%%%%%%%%%%%%%%%%%%
We will need a similar compatibility result for the composition 
between the Neumann trace and resolvents of the Dirichlet Laplacian. 
To state it, we recall the restriction operator $R_{\Om}$ in \eqref{jk-10}.
Also, we denote by $I_{\bbR^n}$ the identity operator (for spaces of functions
defined in $\bbR^n$). Finally, we recall the space \eqref{2.88X} and the 
ultra weak Neumann trace operator $\wti \gamma_{\cN}$ in \eqref{2.8X}, \eqref{2.9X}.

%%%%%%%%%%%%%%%%%%%%%%%%%%%%%
\begin{lemma}\label{new-L1}
Assume Hypothesis \ref{h2.1}. Then
\begin{eqnarray}\label{Obt.1P}
\big(\big(-\wti\Delta_{D,\Omega}-z\wti{I}_{\Omega}\big)^{-1}\circ R_{\Omega} 
,I_{\bbR^n}\big):(H^1(\Omega))^*\to  W_z(\Om),
\quad  z\in\bbC\backslash \si(-\Delta_{D,\Omega}),
\end{eqnarray}
is a well-defined, linear and bounded operator. Consequently, 
\begin{eqnarray}\label{Obt.2P}
\wti\gamma_{\cN}
\big(\big(-\wti\Delta_{D,\Omega}-z\wti{I}_{\Omega}\big)^{-1}\circ R_{\Omega} 
,I_{\bbR^n}\big)\in\cB\bigl((H^1(\Omega))^*,H^{-1/2}(\partial\Omega)\bigr),
\quad z\in\bbC\backslash \si(-\Delta_{D,\Omega}),
\end{eqnarray}
and, hence, 
\begin{eqnarray}\label{Obt.3P}
\big[\wti\gamma_{\cN}
\big(\big(-\wti\Delta_{D,\Omega}-z\wti{I}_{\Omega}\big)^{-1}\circ R_{\Omega} 
,I_{\bbR^n}\big)\big]^*
\in\cB\bigl(H^{1/2}(\partial\Omega),H^1(\Omega)\bigr),\quad
z\in\bbC\backslash \si(-\Delta_{D,\Omega}).
\end{eqnarray}
Furthermore, the operators \eqref{Obt.2P}, \eqref{Obt.3P} are
compatible with \eqref{3.23} and \eqref{3.24}, respectively, in the
sense that for each $z\in\bbC\backslash \si(-\Delta_{D,\Omega})$, 
\begin{equation}\label{Obt.3q}
\wti\gamma_N(-\Delta_{D,\Omega}-zI_{\Omega})^{-1}f
=\wti\gamma_{\cN}
\big(\big(-\wti\Delta_{D,\Omega}-z\wti{I}_{\Omega}\big)^{-1}\circ R_{\Omega} 
,I_{\bbR^n}\big)f\mbox{ in }H^{-1/2}(\partial\Omega),\,\,f\in\LOm,
\end{equation}
and 
\begin{align} \label{Obt.3bq}
\begin{split}
& \big[\wti\gamma_N(-\Delta_{D,\Omega}-zI_{\Omega})^{-1}\big]^*f
=\big[\wti\gamma_{\cN}
\big(\big(-\wti\Delta_{D,\Omega}-z\wti{I}_{\Omega}\big)^{-1}\circ R_{\Omega} 
,I_{\bbR^n}\big)\big]^*f \, \text{ in }\LOm,  \\
& \hspace*{7.5cm} \text{for every element }\, f\in H^{1/2}(\partial\Omega). 
\end{split} 
\end{align}
\end{lemma}
%%%%%%%%%%%%%%%%%%%

This yields the following $L^2$-version of Lemma \ref{lA.3BB}. 

%%%%%%%%%%%%%%%%%%%%%%%%
\begin{lemma}  \lb{l4.8}
Assume Hypothesis \ref{h3.1} and suppose that 
$z\in\bbC\backslash(\si(-\Delta_{\Theta,\Om})\cup\si(-\Delta_{D,\Om}))$.
Then the following resolvent relation holds on $\LOm$, 
\begin{align}
\begin{split}
(-\Delta_{\Theta,\Om}-zI_\Om)^{-1} 
&= (-\Delta_{D,\Om}-zI_\Om)^{-1} 
+ \big[\ga_D (-\Delta_{\Theta,\Om}-{\ol z} I_\Om)^{-1}\big]^* 
\big[\wti\gamma_N (-\Delta_{D,\Om}-zI_\Om)^{-1}\big]   \lb{Na1C}  \\
&= (-\Delta_{D,\Om}-zI_\Om)^{-1} 
+ \big[\wti\gamma_{\cN} (-\Delta_{D,\Om}-{\ol z}I_\Om)^{-1}\big]^*
\big[\ga_D (-\Delta_{\Theta,\Om}-z I_\Om)^{-1}\big]. 
\end{split}
\end{align} 
\end{lemma}
%%%%%%%%%%%%%%%%%%%%%%%%

We note that the special case $\Theta =0$ in Lemma \ref{l4.8} was 
discussed by Nakamura \cite{Na01} (in connection with cubic boxes $\Omega$) 
and subsequently in \cite[Lemma A.3]{GLMZ05} (in the case of a Lipschitz 
domain with a compact boundary).  

We also recall the following useful result.

%%%%%%%%%%%%%%%%%%%%%%%%
\begin{lemma} \lb{lA.3CC}
Assume Hypothesis \ref{h3.1bis} and suppose that 
$z\in\bbC\backslash\si(-\Delta_{\Theta,\Om})$. Then
\begin{equation} \lb{NaD}
\big[\wti M_{\Theta,D,\Om}^{(0)}(z)\big]^* = \wti M_{\Theta,D,\Om}^{(0)}(\ol{z})
\end{equation}
as operators in $\cB\big(H^{-1/2}(\partial\Omega);H^{1/2}(\partial\Omega)\big)$.
In particular, assuming Hypothesis \ref{h3.1}, then 
\begin{equation} \lb{NaDL2}
\big[M_{\Theta,D,\Om}^{(0)}(z)\big]^* = M_{\Theta,D,\Om}^{(0)}(\ol{z}). 
\end{equation}
\end{lemma}
%%%%%%%%%%%%%%%%%%%%%%%%

Next we briefly recall the Herglotz property of the Robin-to-Dirichlet map.
 We recall that an operator-valued function $M(z)\in\cB(\cH)$, $z\in\bbC_+$ 
(where $\bbC_+=\{z\in\bbC\,|\, \Im(z)>0$), for some separable complex Hilbert 
space $\cH$, is called an {\it operator-valued Herglotz function} if 
$M(\dott)$ is analytic on $\bbC_+$ and
\begin{equation}
\Im(M(z)) \ge 0, \quad z\in\bbC_+.  \lb{4.41}
\end{equation} 
Here, as usual, $\Im(M)=(M-M^*)/(2i)$.

%%%%%%%%%%%%%%%%%%%%%%%%
\begin{lemma} \lb{l4.13}
Assume Hypothesis \ref{h3.1bis} and suppose that $z\in\bbC_+$. Then
for every $g\in H^{-1/2}(\dOm)$, $g \neq 0$, one has 
\begin{equation}\lb{4.42}  
\f{1}{2i}\big\langle\ g,\big[\wti M_{\Theta,D}(z) 
- \wti M_{\Theta,D}(z)^*\big]g\big\rangle_{1/2}=\Im(z) \|u_{\Theta}\|^2_{\LOm}
> 0, 
\end{equation}
where $u_{\Theta}$ satisfies
\begin{equation}  
\begin{cases}
(-\Delta - z)u = 0 \text{ in }\,\Om,\quad u \in H^{1}(\Om), \\
\big(\wti\ga_N + \wti \Theta \gamma_D\big) u = g \text{ on } \,\dOm.   \lb{4.43}
 \end{cases}
\end{equation}
In particular, assuming Hypothesis \ref{h3.1}, then 
\begin{equation} \lb{4.44}
\Im\big(M_{\Theta,D,\Om}^{(0)}(z)\big) \ge 0, \quad z\in\bbC_+,
\end{equation}
and hence $M_{\Theta,D,\Om}^{(0)}(\dott)$ is an operator-valued Herglotz 
function on $\LdOm$.
\end{lemma}
%%%%%%%%%%%%%%%%%%%%%%%%
The following result represents a first variant of Krein's resolvent 
formula relating $ \wti \Delta_{\Theta,\Om}$ and $\wti \Delta_{D,\Om}$ 
recently proved in \cite{GM08}:

%%%%%%%%%%%%%%%%%%%%%%%%
\begin{theorem} \lb{tA.3LL}
Assume Hypothesis \ref{h3.1bis} and suppose that 
$z\in\bbC\backslash(\si(-\Delta_{\Theta,\Om})\cup\si(-\Delta_{D,\Om}))$.
Then the following Krein formula holds on $\bigl(H^1(\Omega)\bigr)^*$, 
\begin{align} \lb{NaK2}
\begin{split}
& \big(- \wti \Delta_{\Theta,\Om}-z \wti I_\Om\big)^{-1} 
= \big(- \wti \Delta_{D,\Om}-z \wti I_\Om\big)^{-1}\circ R_{\Om}   
\\ 
& \quad +\big[\wti\gamma_{\cN}
\big(\big(-\wti\Delta_{D,\Om}-\ol{z}\wti I_\Om\big)^{-1}\circ R_{\Om}, 
I_{\bbR^n}\big)\big]^*\wti M_{\Theta,D,\Om}^{(0)}(z)\big[\wti\gamma_{\cN}
\big(\big(-\wti\Delta_{D,\Om}-z\wti I_\Om\big)^{-1}\circ R_{\Om}, 
I_{\bbR^n}\big)\big].
\end{split}
\end{align}
\end{theorem}
%%%%%%%%%%%%%%%%%%%%%%%%
The following result details the $\LOm$-variant of Krein's formula: 

%%%%%%%%%%%%%%%%%%%%%%%%
\begin{theorem} \lb{t4.11}
Assume Hypothesis \ref{h3.1} and suppose that 
$z\in\bbC\backslash(\si(-\Delta_{\Theta,\Om})\cup\si(-\Delta_{D,\Om}))$.
Then the following Krein formula holds on $\LOm$:  
\begin{align} \lb{NaK6}
\begin{split}
(- \Delta_{\Theta,\Om}-zI_\Om)^{-1} &= (- \Delta_{D,\Om}-zI_\Om)^{-1}   \\ 
& \quad +\big[\wti\gamma_N(-\Delta_{D,\Om}-\ol{z}I_\Om)^{-1}\big]^* 
M_{\Theta,D,\Om}^{(0)}(z)
\big[\wti\gamma_N (- \Delta_{D,\Om}-zI_\Om)^{-1}\big].
\end{split}
\end{align} 
\end{theorem}
%%%%%%%%%%%%%%%%%%%%%%%%
It should be noted that, by Lemma \ref{L-Bbb}, the composition of operators 
in the right-hand side of \eqref{NaK6} acts in a well-defined manner on $\LOm$.

An attractive feature of the Krein-type formula \eqref{NaK6} lies in 
the fact that $M_{\Theta,D,\Om}^{(0)}(z)$ encodes spectral information about 
$\Delta_{\Theta,\Om}$. This will be pursued in future work. 

Assuming Hypothesis \ref{h2.1}, the special case $\Theta =0$ then 
connects the Neumann and Dirichlet resolvents, 
\begin{align} \lb{NDK2}
& \big(- \wti \Delta_{N,\Om}-z \wti I_\Om\big)^{-1} 
= \big(- \wti \Delta_{D,\Om}-z \wti I_\Om\big)^{-1}\circ R_{\Om}   
\no\\ 
& \quad +\big[\wti\gamma_{\cN}
\big(\big(-\wti\Delta_{D,\Om}-\ol{z}\wti I_\Om\big)^{-1}\circ R_{\Om}, 
I_{\bbR^n}\big)\big]^*\wti M_{N,D,\Om}^{(0)}(z)\big[\wti\gamma_{\cN}
\big(\big(-\wti\Delta_{D,\Om}-z\wti I_\Om\big)^{-1}\circ R_{\Om}, 
I_{\bbR^n}\big)\big], \\
& \hspace*{8.74cm}
z\in\bbC\backslash(\si(-\Delta_{N,\Om})\cup\si(-\Delta_{D,\Om})), \no
\end{align}
on $\bigl(H^1(\Omega)\bigr)^*$, and similarly,
\begin{align} \lb{NDK3}
(- \Delta_{N,\Om}-zI_\Om)^{-1} &= (- \Delta_{D,\Om}-z I_\Om)^{-1}
\nonumber \\ 
& \quad +\big[\wti\gamma_N(-\Delta_{D,\Om}-\ol{z} I_\Om)^{-1}\big]^* 
M_{N,D,\Om}^{(0)}(z)
\big[\wti\gamma_N (- \Delta_{D,\Om}-z I_\Om)^{-1}\big], \\ 
& \hspace*{4.68cm}
z\in\bbC\backslash(\si(-\Delta_{N,\Om})\cup\si(-\Delta_{D,\Om})), \no 
\end{align}
on $\LOm$. Here $\wti M_{N,D,\Om}^{(0)}(z)$ and $M_{N,D,\Om}^{(0)}(z)$ denote 
the corresponding Neumann-to-Dirichlet operators. 

Due to the fundamental importance of Krein-type resolvent formulas 
(and more generally, Robin-to-Dirichlet maps) in connection with the 
spectral and inverse spectral theory of ordinary and partial differential 
operators, abstract versions, connected to boundary value spaces 
(boundary triples) and self-adjoint extensions of closed symmetric operators 
with equal (possibly infinite) deficiency spaces, have received enormous 
attention in the literature. In particular, we note that Robin-to-Dirichlet 
maps in the context of ordinary differential operators reduce to the 
celebrated (possibly, matrix-valued) Weyl--Titchmarsh function, the basic 
object of spectral analysis in this context.  Since it is impossible to 
cover the literature in this paper, we refer, for instance, to 
\cite[Sect.\ 84]{AG93}, \cite{ADKK07}, \cite{AT03}, \cite{AT05}, \cite{BL07}, 
\cite{BMT01}, \cite{BT04}, \cite{BGW08}, \cite{BMNW08}, \cite{BGP07}, \cite{GMT98}, 
\cite{GM09}, \cite[Ch.\ 13]{Gr09}, \cite{KK02}, \cite{Ko00}--\cite{LT77}, \cite{Ma92}, \cite{MM06}, 
\cite{MPP07}, \cite{Ne83}--\cite{Po08}, \cite{Sa65}, \cite{St50}--\cite{St70a}, 
and the references cited therein. We add, however, that the case of infinite 
deficiency indices in the context of partial differential operators 
(in our concrete case, related to the deficiency indices of the operator 
closure of $-\Delta\upharpoonright_{C^\infty_0(\Om)}$ in $\LOm$), is much less 
studied and the results obtained in this section, especially, under the 
assumption of Lipschitz (i.e., minimally smooth) domains, to the best of our 
knowledge, are new.

Finally, we emphasize once more that Remark \ref{r3.6} 
also applies to the content of this section (assuming that $V$ is 
real-valued in connection with Lemmas \ref{lA.3CC} and \ref{l4.13}).

%%%%%%%%%%%%%%%%%%%%%%%%%%%%%%%%%%%%%%%
%%%%%%%%%%%%%%%%%%%%%%%%%%%%%%%%%%%%%%%
\section{Some Variants of Krein's Resolvent Formula Involving 
Robin-to-Robin Maps} \label{s5}
%%%%%%%%%%%%%%%%%%%%%%%%%%%%%%%%%%%%%%%
%%%%%%%%%%%%%%%%%%%%%%%%%%%%%%%%%%%%%%%

In this section we present our principal results, variants of 
Krein's formula for the difference of resolvents of generalized Robin 
Laplacians corresponding to two different Robin boundary conditions 
on bounded Lipschitz domains. To the best of our knowledge, the results 
in this section are new.

%%%%%%%%%%%%%%%%%%%% 
\begin{hypothesis} \lb{2T.1}
Assume that the conditions in Hypothesis \ref{h2.2} are satisfied by two 
sesquilinear forms $a_{\Theta_1},a_{\Theta_2}$ and, in addition, 
\begin{equation}\lb{3.5ab}
\wti\Theta_1,\wti\Theta_2
\in\cB_{\infty}\big(H^{1/2}(\dOm),H^{-1/2}(\partial\Omega)\big).
\end{equation}
\end{hypothesis}
%%%%%%%%%%%%%%%%%%% 

%%%%%%%%%%%%%%%%%%% 
\begin{lemma} \lb{lA.3BZ}
Assume Hypothesis \ref{2T.1} and suppose that $z\in\bbC\backslash
(\si(-\Delta_{\Theta_1,\Om})\cup\si(-\Delta_{\Theta_2,\Om}))$.
Then the following resolvent relation holds on $\bigl(H^1(\Omega)\bigr)^*$, 
\begin{align}\lb{Na1BZ}
\begin{split}
 \big(-\wti\Delta_{\Theta_1,\Om}-z\wti I_\Om\big)^{-1} 
&=\big(-\wti\Delta_{\Theta_2,\Om}-z\wti I_\Om\big)^{-1}  \\
& \quad 
+ \big(-\wti\Delta_{\Theta_1,\Om}-z\wti I_\Om\big)^{-1}
\ga_D^*\big(\wti\Theta_1-\wti\Theta_2\big)\gamma_D 
\big(-\wti\Delta_{\Theta_2,\Om}-z\wti I_\Om\big)^{-1}.
\end{split}
\end{align}
\end{lemma}
%%%%%%%%%%%%%%%%%%%%%%%
\begin{proof}
To set the stage, we recall \eqref{2.8} and \eqref{2.9}.
Together with \eqref{fcH} and \eqref{fcH3}, these ensure that 
the composition of operators appearing on the right-hand side of 
\eqref{Na1B} is well-defined. Next, let $\phi_1,\phi_2\in L^2(\Om;d^nx)$ 
be arbitrary and define
\begin{align}
\begin{split}
\psi_1 & =(-\Delta_{\Theta_1,\Om}-\ol{z}I_\Om)^{-1}\phi_1 
\in\dom(\Delta_{\Theta_1,\Om})\subset H^{1}(\Om),
\\
\psi_2 & =(-\Delta_{\Theta_2,\Om}-zI_\Om)^{-1}\phi_2 
\in\dom(\Delta_{\Theta_2,\Om})\subset H^{1}(\Om).
\end{split} \lb{Na2Z}
\end{align}
As a consequence of our earlier results, 
both sides of \eqref{Na1B} are bounded operators from $(H^1(\Om))^*$ 
into $H^1(\Om)$. Since $L^2(\Om;d^nx)\hookrightarrow\bigl(H^1(\Om)\bigr)^*$ 
densely, it therefore suffices to show that the following identity holds:
\begin{align}
\begin{split}
&(\phi_1,(-\Delta_{\Theta_1,\Om}-zI_\Om)^{-1}\phi_2)_{L^2(\Om;d^nx)} 
-(\phi_1,(-\Delta_{\Theta_2,\Om}-zI_\Om)^{-1}\phi_2)_{L^2(\Om;d^nx)}
\\
&\quad =\big(\phi_1,(-\Delta_{\Theta_1,\Om}-zI_\Om)^{-1}
\ga_D^*\big(\wti\Theta_1-\wti\Theta_2\big)\gamma_D
(-\Delta_{\Theta_2,\Om}-zI_\Om)^{-1}\phi_2\big)_{L^2(\Om;d^nx)}.
\end{split}
\end{align}
We note that according to \eqref{Na2Z} one has,
\begin{align}
(\phi_1,(-\Delta_{\Theta_1,\Om}-zI_\Om)^{-1}\phi_2)_{L^2(\Om;d^nx)}
&= ((-\Delta_{\Theta_1,\Om}-\ol{z}I_\Om)\psi_1,\psi_2)_{L^2(\Om;d^nx)},
\\
(\phi_1,(-\Delta_{\Theta_2,\Om}-zI_\Om)^{-1}\phi_2)_{L^2(\Om;d^nx)}
&=\big(\big((-\Delta_{\Theta_2,\Om}-zI_\Om)^{-1}\big)^*\phi_1,\phi_2\big)
_{L^2(\Om;d^nx)} \no
\\
&=((-\Delta_{\Theta_2,\Om}-\ol{z}I_\Om)^{-1}\phi_1,\phi_2)_{L^2(\Om;d^nx)}
\no
\\
&=(\psi_1,(-\Delta_{\Theta_2,\Om}-zI_\Om)\psi_2)_{L^2(\Om;d^nx)},
\end{align}
and, further, 
\begin{align}
&\big(\phi_1,(-\Delta_{\Theta_1,\Om}-zI_\Om)^{-1}
\ga_D^*\big(\wti\Theta_1-\wti\Theta_2\big)\gamma_D
(-\Delta_{\Theta_2,\Om}-zI_\Om)^{-1}\phi_2\big)_{L^2(\Om;d^nx)} \no
\\
&\quad ={}_{H^1(\Om)}\big\langle{(-\Delta_{\Theta_1,\Om}-\ol{z}I_\Om)^{-1}\phi_1},
\ga_D^*\big(\wti\Theta_1-\wti\Theta_2\big)\gamma_D
(-\Delta_{\Theta_2,\Om}-zI_\Om)^{-1}\phi_2\big\rangle_{(H^1(\Om))^*} \no
\\
&\quad = \big\langle{\ga_D(-\Delta_{\Theta_1,\Om}-\ol{z}I_\Om)^{-1}\phi_1},
\big(\wti\Theta_1-\wti\Theta_2\big)\gamma_D(-\Delta_{\Theta_2,\Om}-zI_\Om)^{-1}\phi_2
\big\rangle_{1/2} 
\no \\
&\quad =\big\langle{\ga_D\psi_1},\big(\wti\Theta_1-\wti\Theta_2\big)\gamma_D\psi_2
\big\rangle_{1/2}.
\end{align}
Thus, matters have been reduced to proving that
\begin{align}\lb{Na3Z}
\begin{split}
& ((-\Delta_{\Theta_1,\Om}-\ol{z}I_\Om)\psi_1,\psi_2)_{L^2(\Om;d^nx)} 
- (\psi_1,(-\Delta_{\Theta_2,\Om}-zI_\Om)\psi_2)_{L^2(\Om;d^nx)}  \\
& \quad 
=\big\langle\ga_D\psi_1,\big(\wti\Theta_1-\wti\Theta_2\big)\gamma_D\psi_2
\big\rangle_{1/2}.
\end{split}
\end{align}
Using \eqref{wGreen} for the left-hand side of \eqref{Na3Z} one obtains
\begin{align}
& ((-\Delta_{\Theta_1,\Om}-\ol{z}I_\Om)\psi_1,\psi_2)_{L^2(\Om;d^nx)} 
-(\psi_1,(-\Delta_{\Theta_2,\Om}-zI_\Om)\psi_2)_{L^2(\Om;d^nx)} \no 
\\
&\quad = -(\Delta\psi_1,\psi_2)_{L^2(\Om;d^nx)} 
+(\psi_1,\Delta\psi_2)_{L^2(\Om;d^nx)} \no
\\
&\quad = (\nabla\psi_1,\nabla\psi_2)_{L^2(\Om;d^nx)^n} 
-\langle\wti\gamma_N\psi_1,\ga_D\psi_2\rangle_{1/2} 
-(\nabla\psi_1,\nabla\psi_2)_{L^2(\Om;d^nx)^n} 
+\langle\ga_D\psi_1,\wti\gamma_N\psi_2\rangle_{1/2} \no
\\
&\quad = -\langle\wti\gamma_N\psi_1,\ga_D\psi_2\rangle_{1/2} 
+\langle\ga_D\psi_1,\wti\gamma_N\psi_2\rangle_{1/2}. 
\end{align}
Observing that $\wti\gamma_N\psi_j=-\wti\Theta_j\gamma_D\psi_j$ 
since $\psi_j\in\dom(\Delta_{\Theta_j,\Om})$, $j=1,2$, one concludes 
\eqref{Na3Z}.
\end{proof}
%%%%%%%%%%%%%%%%%%%%%%%%

Assuming Hypothesis \ref{2T.1} 
we now introduce the Robin-to-Robin map $\wti M_{\Theta_1,\Theta_2,\Om}^{(0)}(z)$ 
as follows,
\begin{align}\lb{3.44Z}
\wti M_{\Theta_1,\Theta_2,\Om}^{(0)}(z) \colon
\begin{cases}
H^{-1/2}(\dOm) \to H^{-1/2}(\dOm),  \\
\hspace*{1.5cm} f \mapsto -\big(\wti\ga_N +\wti\Theta_2\gamma_D\big)u_{\Theta_1},
\end{cases}  \quad z\in\bbC\backslash\si(-\Delta_{\Theta_1,\Om}), 
\end{align}
where $u_{\Theta_1}$ is the unique solution of
\begin{align}\lb{3.45Z}
(-\Delta-z)u = 0 \,\text{ in }\Om, \quad u \in
H^{1}(\Om), \quad \big(\wti\ga_N +\wti\Theta_1\gamma_D\big)u= f \,\text{ on }\dOm.   
\end{align}

%%%%%%%%%%%%%%
\begin{theorem} \lb{l3.5Z} 
Assume Hypothesis \ref{2T.1}. Then 
\begin{equation}\lb{3.46Z}
\wti M_{\Theta_1,\Theta_2,\Om}^{(0)}(z)\in \cB\big(H^{-1/2}(\dOm)\big), \quad
z\in\bbC\backslash\si(-\Delta_{\Theta_1,\Om}),   
\end{equation}
and 
\begin{equation}\lb{3.47Z}
\wti M_{\Theta_1,\Theta_2,\Om}^{(0)}(z) 
= -I_{\partial\Omega}+\big(\wti\Theta_1-\wti\Theta_2\big)\wti M_{\Theta_1,D,\Om}^{(0)}(z),
\quad z\in\bbC\backslash\si(-\Delta_{\Theta_1,\Om}). 
\end{equation}
In particular, 
\begin{equation}\lb{3.aZ}
\wti M_{\Theta_1,\Theta_2,\Om}^{(0)}(z)\big(\wti\Theta_1-\wti\Theta_2\big)
= -\big(\wti\Theta_1-\wti\Theta_2\big)
+\big(\wti\Theta_1-\wti\Theta_2\big)\wti M_{\Theta_1,D,\Om}^{(0)}(z)
\big(\wti\Theta_1-\wti\Theta_2\big),
\quad z\in\bbC\backslash\si(-\Delta_{\Theta_1,\Om}),
\end{equation}
and 
\begin{equation}\lb{3.50Z}
\big[\wti M_{\Theta_1,\Theta_2,\Om}^{(0)}(z)\big(\wti\Theta_1-\wti\Theta_2\big)\big]^*
=\wti M_{\Theta_1,\Theta_2,\Om}^{(0)}(\ol{z})\big(\wti\Theta_1-\wti\Theta_2\big),
\quad z\in\bbC\backslash\si(-\Delta_{\Theta_1,\Om}).    
\end{equation}
Also, if $z\in\bbC\backslash
(\si(-\Delta_{\Theta_1,\Om})\cup\si(-\Delta_{\Theta_2,\Om}))$, then
\begin{equation}\lb{3.53Z}  
\wti M_{\Theta_1,\Theta_2,\Om}^{(0)}(z)=\wti M_{\Theta_2,\Theta_1,\Om}^{(0)}(z)^{-1}.  
\end{equation}
\end{theorem}
%%%%%%%%%%%%%%
\begin{proof}
The membership in \eqref{3.46Z} is a consequence of \eqref{3.44Z}
and Theorem \ref{t3.2bis}. To see \eqref{3.47Z}, assume that 
$f\in H^{-1/2}(\dOm)$ and denote by $u_{\Theta_1}\in H^{1}(\Om)$ 
the unique function satisfying $(-\Delta-z)u_{\Theta_1}= 0$ in $\Om$
and $(\wti\ga_N +\wti\Theta_1\gamma_D)u_{\Theta_1}=f$ on $\dOm$. Then 
\begin{align}\lb{Kr-1}
\wti M_{\Theta_1,\Theta_2,\Om}^{(0)}(z)f
&= -\big(\wti\ga_N +\wti\Theta_2\gamma_D\big)u_{\Theta_1}
=-\big(\wti\ga_N +\wti\Theta_1\gamma_D\big)u_{\Theta_1}
+\big(\wti\Theta_1-\wti\Theta_2\big)\gamma_Du_{\Theta_1}
\nonumber\\
&= -f+\big(\wti\Theta_1-\wti\Theta_2\big)\wti M_{\Theta_1,D,\Om}^{(0)}(z)f,
\end{align}
proving \eqref{3.47Z}. Going further, \eqref{3.aZ} is a direct consequence 
of \eqref{3.47Z}, and \eqref{3.50Z} is clear from \eqref{3.aZ} and 
Lemma \ref{lA.3CC}. Finally, as far as \eqref{3.53Z} is concerned, 
if $f\in H^{-1/2}(\dOm)$ and $u_{\Theta_2}\in H^{1}(\Om)$ is 
the unique function satisfying $(-\Delta-z)u_{\Theta_2}= 0$ in $\Om$
and $\big(\wti\ga_N +\wti\Theta_2\gamma_D\big)u_{\Theta_2}=f$ on $\dOm$, then 
$\wti M_{\Theta_2,\Theta_1,\Om}^{(0)}(z)f=
-\big(\wti\ga_N +\wti\Theta_1\gamma_D\big)u_{\Theta_2}$. As a consequence, if 
$u_{\Theta_1}\in H^{1}(\Om)$ is the unique function satisfying 
$(-\Delta-z)u_{\Theta_2}= 0$ in $\Om$ and 
$\big(\wti\ga_N +\wti\Theta_1\gamma_D\big)u_{\Theta_1}=
-\big(\wti\ga_N +\wti\Theta_1\gamma_D\big)u_{\Theta_2}$ on $\dOm$, it follows that
$u_{\Theta_2}=-u_{\Theta_1}$ so that 
$\wti M_{\Theta_1,\Theta_2,\Om}^{(0)}(z)
\wti M_{\Theta_2,\Theta_1,\Om}^{(0)}(z)f=
-\big(\wti\ga_N +\wti\Theta_2\gamma_D\big)u_{\Theta_1}
=\big(\wti\ga_N +\wti\Theta_2\gamma_D\big)u_{\Theta_2}=f$. 
In a similar fashion, $\wti M_{\Theta_2,\Theta_1,\Om}^{(0)}(z)
\wti M_{\Theta_1,\Theta_2,\Om}^{(0)}(z)f=f$, so \eqref{3.53Z} is proved. 
\end{proof}
%%%%%%%%%%%%%%%%%%%%%%%%

%%%%%%%%%%%%%%%%%%%%%%%%
\begin{theorem} \lb{lA.3LLZ}
Assume Hypothesis \ref{2T.1} and suppose that $z\in\bbC\backslash
(\si(-\Delta_{\Theta_1,\Om})\cup\si(-\Delta_{\Theta_2,\Om}))$.
Then the following Krein formula holds: 
\begin{align}\lb{NaK2Z}
\begin{split}
& 
(-\wti\Delta_{\Theta_1,\Om}-z\wti I_\Om)^{-1}
=(-\wti\Delta_{\Theta_2,\Om}-z\wti I_\Om)^{-1}
\\[4pt]
& \quad +\big[\gamma_D(-\wti\Delta_{\Theta_2,\Om}-\ol{z}\wti I_\Om)^{-1}\big]^* 
\big[\big(\wti M_{\Theta_1,\Theta_2,\Om}^{(0)}(z)+I_{\partial\Omega}\big)
\big(\wti\Theta_1-\wti\Theta_2\big)\big]
\big[\gamma_D\big(-\wti\Delta_{\Theta_2,\Om}-z\wti I_\Om\big)^{-1}\big],  
\end{split}
\end{align}
as operators on $\bigl(H^1(\Omega)\bigr)^*$.
\end{theorem}
%%%%%%%%%%%%%%%%%%%%%%%%
\begin{proof}
We first claim that 
\begin{align} \lb{NaQ1}
\begin{split}
&\big(\wti\Theta_1-\wti\Theta_2\big)\gamma_D
\big(-\wti\Delta_{\Theta_1,\Om}-z\wti I_\Om\big)^{-1}
\\ 
& \quad 
=\big(\wti\Theta_1-\wti\Theta_2\big)\gamma_D
\big(-\wti\Delta_{\Theta_1,\Om}-z\wti I_\Om\big)^{-1}
\ga_D^*\big(\wti\Theta_1-\wti\Theta_2\big)\gamma_D 
\big(-\wti\Delta_{\Theta_2,\Om}-z\wti I_\Om\big)^{-1}, 
\end{split} 
\end{align}
as operators in $\cB\bigl((H^1(\Om))^*,H^1(\Om)\bigr)$.
To see this, consider an arbitrary $w\in\bigl(H^1(\Om)\bigr)^*$, then introduce
\begin{eqnarray}\lb{NaQ2}
v=\ga_D^*\big(\wti\Theta_1-\wti\Theta_2\big)\gamma_D 
\big(-\wti\Delta_{\Theta_2,\Om}-z\wti I_\Om\big)^{-1}w\in \bigl(H^1(\Om)\bigr)^*, 
\end{eqnarray}
and observe that, under the identification \eqref{jk-9}, \eqref{ga22} yields
\begin{eqnarray}\lb{yG}
\supp \, (v) \subseteq\dOm.
\end{eqnarray}
As far as \eqref{NaQ1} is concerned, the goal is to show that 
\begin{eqnarray}\lb{NaQ3}
\big(\wti\Theta_1-\wti\Theta_2\big)\gamma_D
\big(-\wti\Delta_{\Theta_1,\Om}-z\wti I_\Om\big)^{-1}w
=\big(\wti\Theta_1-\wti\Theta_2\big)\gamma_D
\big(-\wti\Delta_{\Theta_1,\Om}-z\wti I_\Om\big)^{-1}v.
\end{eqnarray}
To this end, we observe from \eqref{Na1BZ} that 
\begin{eqnarray}\lb{NaQ4}
\big(-\wti\Delta_{\Theta_1,\Om}-z\wti I_\Om\big)^{-1}w 
=\big(-\wti\Delta_{\Theta_2,\Om}-z\wti I_\Om\big)^{-1}w
+\big(-\wti\Delta_{\Theta_1,\Om}-z\wti I_\Om\big)^{-1}v.
\end{eqnarray}
Hence, by linearity, 
\begin{align}\lb{NaQ5}
\begin{split} 
\wti\ga_{\cN}\bigl(\big(-\wti\Delta_{\Theta_1,\Om}-z\wti I_\Om\big)^{-1}w,w\bigr) 
&= \wti\ga_{\cN}\bigl(\big(-\wti\Delta_{\Theta_2,\Om}-z\wti I_\Om\big)^{-1}w,w\bigr)  \\ 
& \quad +\wti\ga_{\cN}\bigl(\big(-\wti\Delta_{\Theta_1,\Om}-z\wti I_\Om\big)^{-1}v,0\bigr).
\end{split}
\end{align}
A word of explanation is in order here: First, by Remark~\ref{r3.VF}, 
$\bigl(\big(-\wti\Delta_{\Theta_j,\Om}-z\wti I_\Om\big)^{-1}w,w\bigr)
\in W_z(\Om)$ for $j=1,2$, so the terms in the first line of \eqref{NaQ5}
are well-defined in $H^{-1/2}(\dOm)$ (cf. \eqref{2.8X}). Second, thanks to 
\eqref{yG}, we have that 
$\bigl(\big(-\wti\Delta_{\Theta_1,\Om}-z\wti I_\Om\big)^{-1}v,0\bigr)\in W_z(\Om)$, 
so the last term in \eqref{NaQ5} is also well-defined in $H^{-1/2}(\dOm)$. 
Next, from the fact that the functions
$\big(-\wti\Delta_{\Theta_j,\Om}-z\wti I_\Om\big)^{-1}w$, $j=1,2$, 
satisfy homogeneous Robin boundary conditions, one infers 
\begin{eqnarray}\lb{NaQ6}
\wti\ga_{\cN}\bigl(\big(-\wti\Delta_{\Theta_j,\Om}-z\wti I_\Om\big)^{-1}w,w\bigr) 
=-\wti\Theta_j\ga_D\big(-\wti\Delta_{\Theta_j,\Om}-z\wti I_\Om\big)^{-1}w,
\quad j=1,2.
\end{eqnarray}
In a similar fashion, 
\begin{align}\lb{NaQ7}
\begin{split} 
\wti\ga_{\cN}\bigl(\big(-\wti\Delta_{\Theta_1,\Om}-z\wti I_\Om\big)^{-1}v,0\bigr)
&= \wti\ga_{\cN}\bigl(\big(-\wti\Delta_{\Theta_1,\Om}-z\wti I_\Om\big)^{-1}v,v\bigr)
-\wti\ga_{\cN}\bigl(0,v\bigr)   \\ 
&= -\wti\Theta_1\ga_D
\big(-\wti\Delta_{\Theta_1,\Om}-z\wti I_\Om\big)^{-1}v-\wti\ga_{\cN}\bigl(0,v\bigr).
\end{split} 
\end{align}
To compute $\wti\ga_{\cN}\bigl(0,v\bigr)$, pick an arbitrary  
$\phi\in H^{1/2}(\dOm)$ and assume that $\Phi\in H^1(\Om)$ is such that 
$\ga_D\Phi=\phi$. Then, based on \eqref{2.9X} and \eqref{NaQ2}, one has  
\begin{align} \lb{NaQ8}
\langle\phi,\wti\ga_{\cN}\bigl(0,v\bigr)\rangle_{1/2}
&= -{}_{H^1(\Om)}\langle \Phi,v\rangle_{(H^1(\Om))^*}
\nonumber\\ 
&= -{}_{H^1(\Om)}\big\langle \Phi, 
\ga_D^*\big(\wti\Theta_1-\wti\Theta_2\big)\gamma_D 
\big(-\wti\Delta_{\Theta_2,\Om}-z\wti I_\Om\big)^{-1}w
\big\rangle_{(H^1(\Om))^*}
\nonumber\\ 
&= -\big\langle \ga_D\Phi,\big(\wti\Theta_1-\wti\Theta_2\big)\gamma_D 
\big(-\wti\Delta_{\Theta_2,\Om}-z\wti I_\Om\big)^{-1}w\big\rangle_{1/2}
\nonumber\\ 
&= -\big\langle\phi,\big(\wti\Theta_1-\wti\Theta_2\big)\gamma_D 
\big(-\wti\Delta_{\Theta_2,\Om}-z\wti I_\Om\big)^{-1}w\big\rangle_{1/2}.
\end{align}

This shows that 
\begin{equation}\lb{NaQ9}
\wti\ga_{\cN}\bigl(0,v\bigr)
=-\big(\wti\Theta_1-\wti\Theta_2\big)\gamma_D 
\big(-\wti\Delta_{\Theta_2,\Om}-z\wti I_\Om\big)^{-1}w.
\end{equation}
By plugging \eqref{NaQ6}, \eqref{NaQ7}, and \eqref{NaQ9} back into \eqref{NaQ5},  
one then arrives at 
\begin{align}\lb{NaQ10}
-\wti\Theta_1\ga_D\big(-\wti\Delta_{\Theta_1,\Om}-z\wti I_\Om\big)^{-1}w
&= -\wti\Theta_2\ga_D\big(-\wti\Delta_{\Theta_2,\Om}-z\wti I_\Om\big)^{-1}w
-\wti\Theta_1\ga_D\big(-\wti\Delta_{\Theta_1,\Om}-z\wti I_\Om\big)^{-1}v
\nonumber \\  
& \quad +\big(\wti\Theta_1-\wti\Theta_2\big)\gamma_D 
\big(-\wti\Delta_{\Theta_2,\Om}-z\wti I_\Om\big)^{-1}w.
\end{align}
Upon recalling from \eqref{NaQ4} that 
\begin{equation}\lb{NaQ11}
\big(-\wti\Delta_{\Theta_2,\Om}-z\wti I_\Om\big)^{-1}w
=\big(-\wti\Delta_{\Theta_1,\Om}-z\wti I_\Om\big)^{-1}w 
-\big(-\wti\Delta_{\Theta_1,\Om}-z\wti I_\Om\big)^{-1}v, 
\end{equation}
now \eqref{NaQ3} readily follows from \eqref{NaQ10}, \eqref{NaQ11} and some 
simple algebra. This finishes the proof of \eqref{NaQ11}.

Next, since (see \eqref{3.52bis}) 
\begin{equation}\lb{3.52Z}
\wti M_{\Theta_1,D,\Om}^{(0)}(z)
=\gamma_D(-\Delta_{\Theta_1,\Om}-zI_\Om)^{-1}\gamma_D^*, 
\quad z\in\bbC\backslash\si(-\Delta_{\Theta_1,\Om}), 
\end{equation} 
we may then transform \eqref{NaQ1} into 
\begin{align} \lb{NaK3Z2}
& \big(\wti\Theta_1-\wti\Theta_2\big)\gamma_D
\big(-\wti\Delta_{\Theta_1,\Om}-zI_\Om\big)^{-1}
\nonumber\\ 
& \quad
= \big(\wti\Theta_1-\wti\Theta_2\big)\wti M_{\Theta_1,D,\Om}^{(0)}(z)
\big(\wti\Theta_1-\wti\Theta_2\big)\gamma_D 
\big(-\wti\Delta_{\Theta_2,\Om}-zI_\Om\big)^{-1}
\nonumber\\ 
& \quad
= \big(\wti M_{\Theta_1,\Theta_2,\Om}^{(0)}(z)+I_{\partial\Omega}\big)
\big(\wti\Theta_1-\wti\Theta_2\big)\gamma_D 
\big(-\wti\Delta_{\Theta_2,\Om}-zI_\Om\big)^{-1},
\end{align}
where the last line is based on \eqref{3.47Z}. Taking adjoints 
in \eqref{NaK3Z2} (written with $\ol{z}$ in place of $z$) then leads to 
\begin{align} \lb{NaK5Z}
& \big(-\wti\Delta_{\Theta_1,\Om}-zI_\Om\big)^{-1}\gamma_D^*
\big(\wti\Theta_1-\wti\Theta_2\big)  
\nonumber\\[4pt]
& \quad
=\big[\gamma_D\big(-\wti\Delta_{\Theta_2,\Om}-\ol{z}I_\Om\big)^{-1}\big]^* 
\big[\big(\wti M_{\Theta_1,\Theta_2,\Om}^{(0)}(\ol{z})+I_{\partial\Omega}\big)
\big(\wti\Theta_1-\wti\Theta_2\big)\big]^*
\nonumber\\ 
& \quad
=\big[\gamma_D\big(-\wti \Delta_{\Theta_2,\Om}-\ol{z}I_\Om\big)^{-1}\big]^*
\big(\wti M_{\Theta_1,\Theta_2,\Om}^{(0)}(z)+I_{\partial\Omega}\big)
\big(\wti\Theta_1-\wti\Theta_2\big),
\end{align}
by \eqref{3.50Z}. Replacing this back in \eqref{Na1BZ} then readily yields
\eqref{NaK2Z}.  
\end{proof}
%%%%%%%%%%%%%%%%%%%%%%%

We are interested in proving an $L^2$-version of Krein's formula 
in Theorem \ref{lA.3LLZ}. This requires the following strengthening of 
Hypothesis \ref{2T.1}.
 
%%%%%%%%%%%%%%%%%%%%%%% 
\begin{hypothesis} \lb{2T.2}
Assume that the conditions in Hypothesis \ref{h2.2} are satisfied by two 
sesquilinear forms $a_{\Theta_1},a_{\Theta_2}$ and suppose in addition that, 
\begin{equation}\lb{3.8ab}
\wti\Theta_1,\wti\Theta_2
\in\cB_{\infty}\big(H^{1}(\dOm),L^{2}(\partial\Omega;d^{n-1}\omega)\big).
\end{equation}
\end{hypothesis}
%%%%%%%%%%%%%%%%%%%%%%%

We recall (cf.\ \eqref{4.3})) that Hypothesis \ref{2T.2} is indeed stronger 
than Hypothesis \ref{2T.1}. 

As a preliminary matter, we first discuss the $L^2$-version of 
Theorem \ref{l3.5Z}.  

%%%%%%%%%%%%%%%%%%%%%%%
\begin{theorem} \lb{TKR-1} 
Assume Hypothesis \ref{2T.2}. Then the Robin-to-Robin map, originally 
consider as an operator  
$\wti M_{\Theta_1,\Theta_2,\Om}^{(0)}(z)\in \cB\big(H^{-1/2}(\dOm)\big)$, 
$z\in\bbC\backslash\si(-\Delta_{\Theta_1,\Om})$, extends $($in a compatible 
fashion\,$)$ to an operator 
\begin{equation}\lb{Kr-2}
M_{\Theta_1,\Theta_2,\Om}^{(0)}(z)\in \cB\big(\LdOm\big), \quad
z\in\bbC\backslash\si(-\Delta_{\Theta_1,\Om}),   
\end{equation}
which, for every $z\in\bbC\backslash\si(-\Delta_{\Theta_1,\Om})$, satisfies 
\begin{eqnarray}\lb{Kr-3}
&& M_{\Theta_1,\Theta_2,\Om}^{(0)}(z) 
=-I_{\partial\Omega}+\big(\wti\Theta_1-\wti\Theta_2\big)M_{\Theta_1,D,\Om}^{(0)}(z),
\\[4pt]
\lb{Kr-4}
&& M_{\Theta_1,\Theta_2,\Om}^{(0)}(z)\big(\wti\Theta_1-\wti\Theta_2\big)
= - \big(\wti\Theta_1-\wti\Theta_2\big)
+ \big(\wti\Theta_1-\wti\Theta_2\big)M_{\Theta_1,D,\Om}^{(0)}(z)
\big(\wti\Theta_1-\wti\Theta_2\big),
\\[4pt]
\lb{Kr-5}
&& \big[M_{\Theta_1,\Theta_2,\Om}^{(0)}(z)\big(\wti\Theta_1-\wti\Theta_2\big)\big]^*
=M_{\Theta_1,\Theta_2,\Om}^{(0)}(\ol{z})\big(\wti\Theta_1-\wti\Theta_2\big).
\end{eqnarray}
Furthermore, if $z\in\bbC\backslash
(\si(-\Delta_{\Theta_1,\Om})\cup\si(-\Delta_{\Theta_2,\Om}))$, then also
\begin{equation}\lb{Kr-6}  
M_{\Theta_1,\Theta_2,\Om}^{(0)}(z)=M_{\Theta_2,\Theta_1,\Om}^{(0)}(z)^{-1}.  
\end{equation}
\end{theorem}
%%%%%%%%%%%%%%
\begin{proof} 
That for each $z\in\bbC\backslash\si(-\Delta_{\Theta_1,\Om})$ the mapping 
$\wti M_{\Theta_1,\Theta_2,\Om}^{(0)}(z)\in \cB\big(H^{-1/2}(\dOm)\big)$ extends 
to an operator $M_{\Theta_1,\Theta_2,\Om}^{(0)}(z)\in \cB\big(\LdOm\big)$ 
is a consequence of Theorem \ref{t3.2} and \eqref{3.8ab}. This justifies 
the claim about \eqref{Kr-2}. Properties \eqref{Kr-3}--\eqref{Kr-6} then 
follow from \eqref{Kr-2}, Theorem \ref{l3.5Z}, and a density argument. 
\end{proof}
%%%%%%%%%%%%%%

With these preparatory results in place we are ready to state and prove 
the following $L^2$-version of Krein's formula. 

%%%%%%%%%%%%%%%%%%%%%%%%
\begin{theorem}\lb{Th-KR.1}
Assume Hypothesis \ref{2T.2} and suppose that $z\in\bbC\backslash
(\si(-\Delta_{\Theta_1,\Om})\cup\si(-\Delta_{\Theta_2,\Om}))$.
Then the following Krein formula holds on $L^2(\Omega;d^nx)$: 
\begin{align}\lb{Kr-a.4}
\begin{split}
& 
(-\Delta_{\Theta_1,\Om}-zI_\Om)^{-1}=(-\Delta_{\Theta_2,\Om}-zI_\Om)^{-1}
\\ 
& \quad +\big[\gamma_D(-\Delta_{\Theta_2,\Om}-\ol{z}I_\Om)^{-1}\big]^* 
\big[\big(M_{\Theta_1,\Theta_2,\Om}^{(0)}(z)+I_{\partial\Omega}\big)
\big(\wti\Theta_1-\wti\Theta_2\big)\big]
\big[\gamma_D(-\Delta_{\Theta_2,\Om}-zI_\Om)^{-1}\big]. 
\end{split}
\end{align}
\end{theorem}
%%%%%%%%%%%%%%%%%%%%%%%%
\begin{proof}
We start by observing that the following operators are well-defined, linear
and bounded: 
\begin{align}\label{Kr-a.5}
& (-\Delta_{\Theta_j,\Om}-zI_\Om)^{-1}\in\cB\bigl(\LOm\bigr),\quad j=1,2,
\\
\label{Kr-a.6}
&  \gamma_D(-\Delta_{\Theta_2,\Om}-zI_\Om)^{-1}
\in\cB\bigl(\LOm,H^1(\dOm)\bigr),
\\ 
\label{Kr-a.7}
& \big(\wti\Theta_1-\wti\Theta_2\big)\in\cB\bigl(H^1(\dOm),\LdOm\bigr),
\\ 
\label{Kr-a.8}
& \big(M_{\Theta_1,\Theta_2,\Om}^{(0)}(z)+I_{\partial\Omega}\big)
\in\cB\big(\LdOm\big),
\\ 
\label{Kr-a.9}
& \big[\gamma_D(-\Delta_{\Theta_2,\Om}-\ol{z}I_\Om)^{-1}\big]^* 
\in\cB\bigl(\LdOm,\LOm)\bigr).
\end{align}
Indeed, \eqref{Kr-a.5} follows from the fact that
$z\in\bbC\backslash
\big(\si(-\Delta_{\Theta_1,\Om})\cup\si(-\Delta_{\Theta_2,\Om})\big)$,
\eqref{Kr-a.6} is covered by \eqref{Kr-a.2}, 
\eqref{Kr-a.7} is taken care of by \eqref{3.8ab}, 
\eqref{Kr-a.8} follows from \eqref{Kr-2}, and
\eqref{Kr-a.9} is a consequence of \eqref{Kr-a.3}.
Altogether, this shows that both sides of \eqref{Kr-a.4} are 
bounded operators on $\LOm$. With this in hand, the desired conclusion follows
from Theorem \ref{lA.3LLZ}, \eqref{Obt.3bis} and the fact that the operators 
\eqref{faH} and \eqref{fcH} are compatible. 
\end{proof}
%%%%%%%%%%%%%%%%%

We conclude by establishing the following Herglotz property 
for the Robin-to-Robin map composed (to the right) by 
$\big(\wti\Theta_1-\wti\Theta_2\big)$. Specifically we have the following result: 

%%%%%%%%%%%%%%%%%
\begin{theorem}\lb{t-MH}
Assume Hypothesis \ref{2T.1} and suppose that
$z\in\bbC\backslash \si(-\Delta_{\Theta_1,\Omega})$.
Then the operator $\wti M_{\Theta_1,\Theta_2}^{(0)}(z)
\big(\wti\Theta_1 - \wti\Theta_2\big)$,
and hence $\big[\wti M_{\Theta_1,\Theta_2}^{(0)}(z)+I_{\Omega}\big]
\big(\wti\Theta_1 - \wti\Theta_2\big)$, has the Herglotz property
when considered as operators in
$\cB\bigl(H^{-1/2}(\partial\Omega)\bigr)$.

Consequently, if Hypothesis \ref{2T.2} is assumed and
$z\in\bbC\backslash \si(-\Delta_{\Theta_1,\Omega})$, then 
$M_{\Theta_1,\Theta_2}^{(0)}(z)\big(\wti\Theta_1 - \wti\Theta_2\big)$ and
$\big[M_{\Theta_1,\Theta_2}^{(0)}(z)+I_{\Omega}\big]
\big(\wti\Theta_1 - \wti\Theta_2\big)$ also have the Herglotz property
when considered as operators in
$\cB\bigl(L^{2}(\partial\Omega;d^{n-1}\omega)\bigr)$. 
\end{theorem}
%%%%%%%%%%%%%%%%%
\begin{proof}
By Theorem \ref{TKR-1} it suffices to prove only the first
part in the statement. To this end, we recall \eqref{3.47Z} in  
Theorem \ref{l3.5Z}. Composing the latter on the right by 
$(\wti\Theta_1 - \wti\Theta_2)$ then yields
\begin{eqnarray}\lb{Mk-1}
\wti M_{\Theta_1,\Theta_2}^{(0)}(z)\big(\wti\Theta_1 - \wti\Theta_2\big)
=-(\wti\Theta_1 - \wti\Theta_2)+\big(\wti\Theta_1 - \wti\Theta_2\big)
\wti M_{\Theta_1,D}^{(0)}(z)\big(\wti\Theta_1 - \wti\Theta_2\big),\quad
z\in\bbC\backslash \si(-\Delta_{\Theta_1,\Omega}).
\end{eqnarray}
Consequently,
\begin{align}\lb{Mk-2}
\Im \big[\wti M_{\Theta_1,\Theta_2}^{(0)}(z)
\big(\wti\Theta_1 - \wti\Theta_2\big)\big]
&= \Im \big[\big(\wti\Theta_1 - \wti\Theta_2\big)
\wti M_{\Theta_1,D}^{(0)}(z)\big(\wti\Theta_1 - \wti\Theta_2\big)\big]
\\ 
&=  \big(\wti\Theta_1 - \wti\Theta_2\big) 
\Im \big[\wti M_{\Theta_1,D}^{(0)}(z)\big]\big(\wti\Theta_1 - \wti\Theta_2\big),
\quad z\in\bbC\backslash \si(-\Delta_{\Theta_1,\Omega}).
\nonumber
\end{align}
Now one can use Lemma \ref{l4.13} in order to conclude that
\begin{equation} \lb{Mk-3}
\Im \big[\wti M_{\Theta_1,\Theta_2}^{(0)}(z)
\big(\wti\Theta_1 - \wti\Theta_2\big)\big]\geq 0,    
\end{equation}
as desired.  
\end{proof}
%%%%%%%%%%%%%%%%%%

We note again that Remark \ref{r3.6} also applies to the content of this 
section (assuming that $V$ is real-valued in connection with \eqref{Kr-5} and 
Theorem \ref{t-MH}).

%%%%%%%%%%%%%%%%%%%%%%%%%%%%%%%%%%%%%%
%%%%%%%%%%%%%%% appendices %%%%%%%%%%%%%%%%
\appendix
%%%%%%%%%%%% appendix A %%%%%%%%%%%%%%
\section{Properties of Sobolev Spaces and \\ Boundary Traces  
for Lipschitz Domains} \lb{sA}
\renewcommand{\theequation}{A.\arabic{equation}}
\renewcommand{\thetheorem}{A.\arabic{theorem}}
\setcounter{theorem}{0} \setcounter{equation}{0}
%%%%%%%%%%%%%%%%%%%%%%%%%%%%%%%%%%%%%%
%%%%%%%%%%%%%%%%%%%%%%%%%%%%%%%%%%%%%%

The purpose of this appendix is to recall some basic facts in connection with 
Sobolev spaces corresponding to Lipschitz domains 
$\Om\subset\bbR^n$, $n\in\bbN$, $n\geq 2$, and their boundaries. 
For more details we refer again to \cite{GM08}.

In this manuscript we use the following notation for the standard
Sobolev Hilbert spaces ($s\in\bbR$),
\begin{align}
H^{s}(\bbR^n)&=\bigg\{U\in \cS(\bbR^n)^\prime \,\bigg|\,
\norm{U}_{H^{s}(\bbR^n)}^2 = \int_{\bbR^n} d^n \xi \, \big|\hatt
U(\xi)\big|^2\big(1+\abs{\xi}^{2s}\big) <\infty \bigg\},
\\
H^{s}(\Om)&=\left\{u\in \cD^\prime(\Om) \,|\, u=U|_\Om \text{
for some } U\in H^{s}(\bbR^n) \right\},
\\
H_0^{s}(\Om) &=\{u\in H^s(\bbR^n)\,|\, \supp\,(u)\subseteq\ol{\Om}\}.
\end{align}
Here $\cD^\prime(\Om)$ denotes the usual set of distributions on
$\Omega\subseteq \bbR^n$, $\Omega$ open and nonempty 
(with $\cD(\Om)$ standing for the space of test functions in $\Om$),
$\cS(\bbR^n)^\prime$ is the space of tempered distributions on
$\bbR^n$, and $\hatt U$ denotes the Fourier transform of $U\in
\cS(\bbR^n)^\prime$. It is then immediate that
\begin{equation}\label{incl-xxx}
H^{s_1}(\Omega)\hookrightarrow H^{s_0}(\Omega) \, \text{ for } \,
-\infty<s_0\leq s_1<+\infty,
\end{equation}
continuously and densely.

Next, we recall the
definition of a Lipschitz-domain $\Omega\subseteq\bbR^n$, $\Om$
open and nonempty, for convenience of the reader: Let ${\mathcal
N}$ be a space of real-valued functions in $\bbR^{n-1}$.  One
calls a bounded domain $\Omega\subset\bbR^n$ of class ${\mathcal
N}$ if there exists a finite open covering $\{{\mathcal
O}_j\}_{1\leq j\leq N}$ of the boundary $\partial\Omega$ of $\Om$
with the property that, for every $j\in\{1,...,N\}$, ${\mathcal
O}_j\cap\Omega$ coincides with the portion of ${\mathcal O}_j$
lying in the over-graph of a function $\varphi_j\in{\mathcal N}$
(considered in a new system of coordinates obtained from the
original one via a rigid motion). If
${\mathcal  N}$ is ${\rm Lip}\,(\bbR^{n-1})$, the space of
real-valued functions satisfying a (global) Lipschitz condition in
$\bbR^{n-1}$, is called a {\it Lipschitz
domain}; cf.\ \cite[p.\ 189]{St70}, where such domains are called
``minimally smooth''. The classical theorem of
Rademacher of almost everywhere differentiability of Lipschitz
functions ensures that, for any  Lipschitz domain $\Omega$, the
surface measure $d^{n-1} \omega$ is well-defined on  $\partial\Omega$ and
that there exists an outward  pointing normal vector $\nu$ at
almost every point of $\partial\Omega$. 

For a Lipschitz domain $\Omega\subset\bbR^n$ it is known that
\begin{equation}\lb{dual-xxx}
\bigl(H^{s}(\Omega)\bigr)^*=H^{-s}(\Omega), \quad - 1/2 <s< 1/2.
\end{equation}
See \cite{Tr02} for this and other related properties. We also refer to 
our convention of using the {\it adjoint} 
(rather than the dual) space $X^*$ of a Banach space $X$ as described near 
the end of the introduction.

Next, assume that $\Omega\subset\bbR^n$ is the domain lying above
the graph of a Lipschitz function $\varphi\colon\bbR^{n-1}\to\bbR$. 
Then for $0\leq s \le 1$, the Sobolev space
$H^s(\partial\Omega)$ consists of functions $f\in
L^2(\partial\Omega;d^{n-1} \omega)$ such that $f(x',\varphi(x'))$,
as a function of $x'\in\bbR^{n-1}$, belongs to $H^s(\bbR^{n-1})$.
In this scenario we set 
\begin{equation}
H^s(\dOm) = \big(H^{-s}(\dOm)\big)^*, \quad -1 \le s \le 0.   \lb{A.6}
\end{equation}
To define $H^s(\dOm)$, $0\leq s \le 1$, when $\Om$ is a Lipschitz domain with 
compact boundary, we use a smooth partition of unity to reduce matters to the 
graph case. More precisely, if $0\leq s\leq 1$ then $f\in H^s(\partial\Omega)$ 
if and only if the assignment 
${\mathbb{R}}^{n-1}\ni x'\mapsto (\psi f)(x',\varphi(x'))$ is in 
$H^s({\mathbb{R}}^{n-1})$ whenever $\psi\in C^\infty_0({\mathbb{R}}^n)$
and $\varphi\colon {\mathbb{R}}^{n-1}\to{\mathbb{R}}$ is a Lipschitz function
with the property that if $\Sigma$ is an appropriate rotation and
translation of $\{(x',\varphi(x'))\in\bbR^n \,|\,x'\in{\mathbb{R}}^{n-1}\}$, 
then $(\supp\, (\psi) \cap\partial\Omega)\subset\Sigma$ (this appears to 
be folklore, but a proof will appear in \cite[Proposition 2.4]{MM07}). 
Then Sobolev spaces with a negative amount of smoothness are defined as 
in \eqref{A.6} above. 

From the above characterization of $H^s(\partial\Omega)$ it follows that 
any property of Sobolev spaces (of order $s\in[-1,1]$) defined in Euclidean 
domains, which are invariant under multiplication by smooth, compactly 
supported functions as well as composition by bi-Lipschitz diffeomorphisms, 
readily extends to the setting of $H^s(\partial\Omega)$ (via localization and
pull-back). As a concrete example, for each Lipschitz domain $\Omega$ 
with compact boundary, one has  
\begin{equation} \label{EQ1}
H^s(\partial\Omega)\hookrightarrow H^{s-\varepsilon}(\partial\Omega)
\, \text{ compactly if }\,0<\varepsilon\leq s\leq 1.  
\end{equation}
For additional background 
information in this context we refer, for instance, to \cite{Au04}, 
\cite{Au06}, \cite[Chs.\ V, VI]{EE89}, \cite[Ch.\ 1]{Gr85}, 
\cite[Ch.\ 3]{Mc00}, \cite[Sect.\ I.4.2]{Wl87}.

Moving on, we next consider the following bounded linear map
\begin{equation}
\begin{cases} \big\{(w,f)\in
L^2(\Omega;d^nx)^n\times \big(H^1(\Omega)\big)^*
\,\big|\,{\rm div}(w)=f|_{\Omega}\big\} \to
H^{-1/2}(\partial\Omega)
=\big(H^{1/2}(\partial\Omega)\big)^* \\
\hspace*{7.5cm} w\mapsto \nu\cdot (w,f)  \end{cases}  \lb{A.11}
\end{equation}
by setting
\begin{equation}
{}_{H^{1/2}(\dOm)}\langle \phi, \nu\cdot (w,f) \rangle_{(H^{1/2}(\dOm)^*}
=\int_{\Omega}d^nx\, \ol{\nabla\Phi(x)} \cdot w(x)
+ {}_{H^1(\Om)}\langle \Phi,f\rangle_{(H^1(\Om))^*}    \lb{A.11a}
\end{equation}
whenever $\phi\in H^{1/2}(\partial\Omega)$ and $\Phi\in
H^{1}(\Omega)$ is such that $\ga_D\Phi=\phi$. Here  
${}_{H^1(\Om)}\langle \Phi,f\rangle_{(H^1(\Om))^*}$ 
in \eqref{A.11a} is the natural pairing  between functionals 
in $\big(H^1(\Omega)\big)^*$ and elements in $H^1(\Omega)$ 
(which, in turn, is compatible with the (bilinear) distributional 
pairing). It should be remarked that the above definition is independent 
of the particular extension $\Phi\in H^{1}(\Omega)$ of $\phi$. 

Going further, one can introduce the ultra weak Neumann trace operator
$\wti\gamma_{\cN}$ as follows:
\begin{equation}\lb{A.16}
\wti\gamma_{\cN}\colon \begin{cases} 
\big\{(u,f)\in H^1(\Om) \times \big(H^1(\Omega)\big)^*
\,\big|\, \Delta u =f|_{\Omega}\big\}\to H^{-1/2}(\dOm)\\ 
\hspace*{6.13cm}  u \mapsto \wti\gamma_{\cN}(u,f)=\nu\cdot(\nabla u,f),  
\end{cases}
\end{equation}
with the dot product  understood in the sense of \eqref{A.11}. We
emphasize that the ultra weak Neumann trace operator $\wti\gamma_{\cN}$ in
\eqref{A.16} is a re-normalization of the operator $\gamma_N$ introduced
in \eqref{2.7} relative to the extension of $\Delta u\in H^{-1}(\Omega)$
to an element $f$ of the space $\big(H^1(\Omega)\big)^*
=\{g\in H^{-1}(\bbR^n)\,|\, \supp \,(g) \subseteq\ol{\Om}\}$. 
For the relationship between the weak and ultra weak Neumann trace operators, 
see \eqref{2.10X}--\eqref{2.12X}. In addition, one can show that 
the ultra weak Neumann trace operator \eqref{A.16} is onto (indeed, this is 
a corollary of Theorem \ref{t3.XV}). We note that \eqref{A.11a} and 
\eqref{A.16} yield the following Green's formula
\begin{equation}
\langle \ga_D\Phi, \wti\gamma_{\cN}(u,f)\rangle_{1/2} 
= (\nabla \Phi, \nabla u)_{\LOm^n} 
+  {}_{H^1(\Om)}\langle \Phi, f\rangle_{(H^1(\Om))^*},
\lb{wGreen}
\end{equation}
valid for any $u\in H^{1}(\Om)$, $f\in \big(H^{1}(\Om)\big)^*$
with $\Delta u=f|_{\Om}$, and any $\Phi\in H^{1}(\Om)$. The
pairing on the left-hand side of \eqref{wGreen} is between
functionals in $\big(H^{1/2}(\dOm)\big)^*$ and elements in
$H^{1/2}(\dOm)$, whereas the last pairing on the right-hand side in \eqref{wGreen} 
is between functionals in $\big(H^{1}(\Om)\big)^*$ and elements in
$H^{1}(\Om)$. For further use, we also note that the adjoint of
\eqref{2.6} maps boundedly as follows
\begin{equation}\lb{ga*}
\ga_D^* \colon \big(H^{s-1/2}(\dOm)\big)^* \to (H^{s}(\Om)\big)^*,
\quad 1/2<s<3/2. 
\end{equation} 
Identifying $\big(H^{s}(\Om)\big)^*$ with 
$H^{-s}_0(\Om)\hookrightarrow H^{-s}(\bbR^n)$ 
(cf. Proposition~2.9 in \cite{JK95}), it follows that 
\begin{equation}\lb{ga22}
{\rm ran}(\ga_D^*)\subseteq\{u\in H^{-s}(\bbR^n)\,|\,
\supp \,(u) \subseteq\dOm\}, \quad 1/2<s<3/2. 
\end{equation} 

%%%%%%%%%%%%%%%%%%%%%%%%%%%
\begin{remark} \lb{rA.4}
While it is tempting to view $\ga_D$ as an unbounded but densely
defined operator on $\LOm$ whose domain contains the space
$C_0^\infty(\Om)$, one should note that in this case its adjoint
$\ga_D^*$ is not densely defined: Indeed (cf.\ \cite[Remark A.4]{GLMZ05}),
 $\dom(\gamma_D^*)=\{0\}$ and hence $\gamma_D$
is not a closable linear operator in $\LOm$.
\end{remark}
%%%%%%%%%%%%%%%%%%%%%%%%%%%

We conclude this appendix by recalling the following result from \cite{GMZ07}.  

%%%%%%%%%%%%%%%%%%%%%%%%%%%
\begin{lemma} [cf.\ \cite{GMZ07}, Lemma A.6] \lb{lA.6x}
Suppose $\Om\subset\bbR^n$, $n\geq 2$, is an open Lipschitz domain with 
a compact, nonempty boundary $\dOm$. Then the Dirichlet trace operator 
$\ga_D$ $($originally considered as in \eqref{2.6}$)$ satisfies 
\eqref{A.62x}. 
\end{lemma}
%%%%%%%%%%%%%%%%%%%%%%%%%%%

%%%%%%%%%%%%%%%%%%%%%%%%%%%%%%%%%%%%%%
%%%%%%%%%%%%%%%%% appendix B %%%%%%%%%%%%%%
\section{Sesquilinear Forms and Associated Operators}
\lb{sB}
\renewcommand{\theequation}{B.\arabic{equation}}
\renewcommand{\thetheorem}{B.\arabic{theorem}}
\setcounter{theorem}{0} \setcounter{equation}{0}
%%%%%%%%%%%%%%%%%%%%%%%%%%%%%%%%%%%%%%
%%%%%%%%%%%%%%%%%%%%%%%%%%%%%%%%%%%%%%

In this appendix we describe a few basic facts on sesquilinear forms and 
linear operators associated with them. A slightly more expanded version 
of this material appeared in \cite[Appendix B]{GM08}. 

Let $\cH$ be a complex separable Hilbert space with scalar product 
$(\dott,\dott)_{\cH}$ (antilinear in the first and linear in the second 
argument), $\cV$ a reflexive Banach space continuously and densely embedded 
into $\cH$. Then also $\cH$ embeds continuously and densely into $\cV^*$. 
\begin{equation}
\cV  \hookrightarrow \cH  \hookrightarrow \cV^*.     \lb{B.1}
\end{equation}
Here the continuous embedding $\cH\hookrightarrow \cV^*$ is accomplished via 
the identification
\begin{equation}
\cH \ni u \mapsto (\dott,u)_{\cH} \in \cV^*,     \lb{B.2}
\end{equation}
and we recall the convention in this manuscript (cf.\ the discussion at 
the end of the introduction) that if $X$ denotes a Banach space, $X^*$ denotes 
the {\it adjoint space} of continuous  conjugate linear functionals on $X$, also 
known as the {\it conjugate dual} of $X$.

In particular, if the sesquilinear form 
\begin{equation}
{}_{\cV}\langle \dott, \dott \rangle_{\cV^*} \colon \cV \times \cV^* \to \bbC
\end{equation}
denotes the duality pairing between $\cV$ and $\cV^*$, then  
\begin{equation}
{}_{\cV}\langle u,v\rangle_{\cV^*} = (u,v)_{\cH}, \quad u\in\cV, \; 
v\in\cH\hookrightarrow\cV^*,   \lb{B.3}
\end{equation}
that is, the $\cV, \cV^*$ pairing 
${}_{\cV}\langle \dott,\dott \rangle_{\cV^*}$ is compatible with the 
scalar product $(\dott,\dott)_{\cH}$ in $\cH$.

Let $T \in\cB(\cV,\cV^*)$. Since $\cV$ is reflexive, $(\cV^*)^* = \cV$, one has
\begin{equation}
T \colon \cV \to \cV^*, \quad  T^* \colon \cV \to \cV^*   \lb{B.4}
\end{equation}
and
\begin{equation}
{}_{\cV}\langle u, Tv \rangle_{\cV^*} 
= {}_{\cV^*}\langle T^* u, v\rangle_{(\cV^*)^*} 
= {}_{\cV^*}\langle T^* u, v \rangle_{\cV} 
= \ol{{}_{\cV}\langle v, T^* u \rangle_{\cV^*}}. 
\end{equation}
{\it Self-adjointness} of $T$ is then defined by $T=T^*$, that is, 
\begin{equation}
{}_{\cV}\langle u,T v \rangle_{\cV^*} 
= {}_{\cV^*}\langle T u, v \rangle_{\cV}
= \ol{{}_{\cV}\langle v, T u \rangle_{\cV^*}}, \quad u, v \in \cV,    \lb{B.5}
\end{equation} 
{\it nonnegativity} of $T$ is defined by 
\begin{equation}
{}_{\cV}\langle u, T u \rangle_{\cV^*} \geq 0, \quad u \in \cV,    \lb{B.6}
\end{equation} 
and {\it boundedness from below of $T$ by $c_T \in\bbR$} is defined by
\begin{equation}
{}_{\cV}\langle u, T u \rangle_{\cV^*} \geq c_T \|u\|^2_{\cH}, 
\quad u \in \cV.   
\lb{B.6a}
\end{equation} 
(By \eqref{B.3}, this is equivalent to 
${}_{\cV}\langle u, T u \rangle_{\cV^*} \geq c_T \,
{}_{\cV}\langle u, u \rangle_{\cV^*}$, $u \in \cV$.)
 
Next, let the sesquilinear form $a(\dott,\dott)\colon\cV \times \cV \to \bbC$ 
(antilinear in the first and linear in the second argument) be 
{\it $\cV$-bounded}, that is, there exists a $c_a>0$ such that 
\begin{equation}
|a(u,v)| \le c_a \|u\|_{\cV} \|v\|_{\cV},  \quad u, v \in \cV. 
\end{equation}
Then $\wti A$ defined by
\begin{equation}
\wti A \colon \begin{cases} \cV \to \cV^*, \\ 
\, v \mapsto \wti A v = a(\dott,v), \end{cases}    \lb{B.7}
\end{equation} 
satisfies
\begin{equation}
\wti A \in\cB(\cV,\cV^*) \, \text{ and } \, 
{}_{\cV}\big\langle u, \wti A v \big\rangle_{\cV^*} 
= a(u,v), \quad  u, v \in \cV.    \lb{B.8}
\end{equation}
Assuming further that $a(\dott,\dott)$ is {\it symmetric}, that is,
\begin{equation}
a(u,v) = \ol{a(v,u)},  \quad u,v\in \cV,    \lb{B.9}
\end{equation} 
and that $a$ is {\it $\cV$-coercive}, that is, there exists a constant 
$C_0>0$ such that
\begin{equation}
a(u,u)  \geq C_0 \|u\|^2_{\cV}, \quad u\in\cV,    \lb{B.10}
\end{equation}
respectively, then,
\begin{equation}
\wti A \colon \cV \to \cV^* \, \text{ is bounded, self-adjoint, and boundedly 
invertible.}    \lb{B.11}
\end{equation}
Moreover, denoting by $A$ the part of $\wti A$ in $\cH$ defined by 
\begin{align}
\dom(A) = \big\{u\in\cV \,|\, \wti A u \in \cH \big\} \subseteq \cH, \quad 
A= \wti A\big|_{\dom(A)}\colon \dom(A) \to \cH,   \lb{B.12} 
\end{align}
then $A$ is a (possibly unbounded) self-adjoint operator in $\cH$ satisfying 
\begin{align}
& A \geq C_0 I_{\cH},   \lb{B.13}  \\
& \dom\big(A^{1/2}\big) = \cV.  \lb{B.14}
\end{align}
In particular, 
\begin{equation}
A^{-1} \in\cB(\cH).   \lb{B.15}
\end{equation}
The facts \eqref{B.1}--\eqref{B.15} are a consequence of the Lax--Milgram 
theorem and the second representation theorem for symmetric sesquilinear forms.
Details can be found, for instance, in \cite[\S VI.3, \S VII.1]{DL00}, 
\cite[Ch.\ IV]{EE89}, and \cite{Li62}.  

Next, consider a symmetric form $b(\dott,\dott)\colon \cV\times\cV\to\bbC$ 
and assume that $b$ is {\it bounded from below by $c_b\in\bbR$}, that is,
\begin{equation}
b(u,u) \geq c_b \|u\|_{\cH}^2, \quad u\in\cV.  \lb{B.19}
\end{equation} 
Introducing the scalar product 
$(\dott,\dott)_{\cV(b)}\colon \cV\times\cV\to\bbC$ 
(with  associated norm $\|\dott\|_{\cV(b)}$) by
\begin{equation}
(u,v)_{\cV(b)} = b(u,v) + (1- c_b)(u,v)_{\cH}, \quad u,v\in\cV,  \lb{B.20}
\end{equation}
turns $\cV$ into a pre-Hilbert space $(\cV; (\dott,\dott)_{\cV(b)})$, 
which we denote by 
$\cV(b)$. The form $b$ is called {\it closed} if $\cV(b)$ is actually 
complete, and hence a Hilbert space. The form $b$ is called {\it closable} 
if it has a closed extension. If $b$ is closed, then
\begin{equation}
|b(u,v) + (1- c_b)(u,v)_{\cH}| \le \|u\|_{\cV(b)} \|v\|_{\cV(b)}, 
\quad u,v\in \cV,   
\lb{B.21}
\end{equation}
and 
\begin{equation}
|b(u,u) + (1 - c_b)\|u\|_{\cH}^2| = \|u\|_{\cV(b)}^2, \quad u \in \cV,   
\lb{B.22}
\end{equation}
show that the form $b(\dott,\dott)+(1 - c_b)(\dott,\dott)_{\cH}$ is a 
symmetric, $\cV$-bounded, and $\cV$-coercive sesquilinear form. Hence, 
by \eqref{B.7} and \eqref{B.8}, there exists a linear map
\begin{equation}
\wti B_{c_b} \colon \begin{cases} \cV(b) \to \cV(b)^*, \\ 
\hspace*{.51cm}  
v \mapsto \wti B_{c_b} v = b(\dott,v) +(1 - c_b)(\dott,v)_{\cH},  
\end{cases}    
\lb{B.23}
\end{equation}
with 
\begin{equation}
\wti B_{c_b} \in\cB(\cV(b),\cV(b)^*) \, \text{ and } \, 
{}_{\cV(b)}\big\langle u, \wti B_{c_b} v \big\rangle_{\cV(b)^*} 
= b(u,v)+(1 -c_b)(u,v)_{\cH}, \quad  u, v \in \cV.    \lb{B.24}
\end{equation}
Introducing the linear map
\begin{equation}
\wti B = \wti B_{c_b} + (c_b - 1)\wti I \colon \cV(b)\to\cV(b)^*,
\lb{B.24a}
\end{equation}
where $\wti I\colon \cV(b)\hookrightarrow\cV(b)^*$ denotes 
the continuous inclusion (embedding) map of $\cV(b)$ into $\cV(b)^*$, one 
obtains a self-adjoint operator $B$ in $\cH$ by restricting $\wti B$ to $\cH$,
\begin{align}
\dom(B) = \big\{u\in\cV \,\big|\, \wti B u \in \cH \big\} \subseteq \cH, \quad 
B= \wti B\big|_{\dom(B)}\colon \dom(B) \to \cH,   \lb{B.25} 
\end{align} 
satisfying the following properties:
\begin{align}
& B \geq c_b I_{\cH},  \lb{B.26} \\
& \dom\big(|B|^{1/2}\big) = \dom\big((B - c_bI_{\cH})^{1/2}\big) 
= \cV,  \lb{B.27} \\
& b(u,v) = \big(|B|^{1/2}u, U_B |B|^{1/2}v\big)_{\cH}    \lb{B.28b} \\
& \hspace*{.97cm} 
= \big((B - c_bI_{\cH})^{1/2}u, (B - c_bI_{\cH})^{1/2}v\big)_{\cH} 
+ c_b (u, v)_{\cH}  
\lb{B.28} \\
& \hspace*{.97cm} 
= {}_{\cV(b)}\big\langle u, \wti B v \big\rangle_{\cV(b)^*},  
\quad u, v \in \cV, \lb{B.28a} \\
& b(u,v) = (u, Bv)_{\cH}, \quad  u\in \cV, \; v \in\dom(B),  \lb{B.29} \\
& \dom(B) = \{v\in\cV\,|\, \text{there exists an $f_v\in\cH$ such that}  \no \\
& \hspace*{3.05cm} b(w,v)=(w,f_v)_{\cH} \text{ for all $w\in\cV$}\},  
\lb{B.30} \\
& Bu = f_u, \quad u\in\dom(B),  \no \\
& \dom(B) \text{ is dense in $\cH$ and in $\cV(b)$}.  \lb{B.31}
\end{align}
Properties \eqref{B.30} and \eqref{B.31} uniquely determine $B$. 
Here $U_B$ in \eqref{B.28} is the partial isometry in the polar 
decomposition of $B$, that is, 
\begin{equation}
B=U_B |B|, \quad  |B|=(B^*B)^{1/2}.   \lb{B.32}
\end{equation}
The operator $B$ is called {\it the operator associated with the form $b$}. 

The facts \eqref{B.19}--\eqref{B.31} comprise the second representation 
theorem of sesquilinear forms (cf.\ \cite[Sect.\ IV.2]{EE89}, 
\cite[Sects.\ 1.2--1.5]{Fa75}, and \cite[Sect.\ VI.2.6]{Ka80}).

\bigskip

%%%%%%%%%%%%%%%%%%
%%%%%%%%%%%%%%%%%%
\noindent {\bf Acknowledgments.}
We wish to thank Gerd Grubb for questioning an inaccurate claim in an 
earlier version of the paper and  Maxim Zinchenko for helpful discussions 
on this topic. Fritz Gesztesy would like to thank all organizers of the 
international conference on Modern Analysis and Applications (MAA 2007), 
and especially, Vadym Adamyan, for their kind invitation, the stimulating 
atmosphere during the meeting, and the hospitality extended to him during 
his stay in Odessa in April of 2007. He is also indebted to 
Vyacheslav Pivovarchik for numerous assistance before and during this 
conference.
%%%%%%%%%%%%%%%%%%
%%%%%%%%%%%%%%%%%%

%%%%%%%%%%%%%%%%%%%%%%%%%%%%%%%%%%%%%%
%%%%%%%%%%%%%%%%%%%%%%%%%%%%%%%%%%%%%%

\end{document}